\documentclass[11pt,a4paper,reqno, draft]{amsart}
\usepackage{amsmath, amssymb, latexsym, enumerate,  bm}

\usepackage{thmtools}
\usepackage{float}

\usepackage[pdftex]{hyperref}
\usepackage{cleveref}
\usepackage{bbm}
\usepackage{cases}
\usepackage{array}
\usepackage{tikz-cd}
\usepackage[all]{xy}

\usepackage[hmarginratio={1:1},vmarginratio={1:1},lmargin=80.0pt,tmargin=90.0pt]{geometry}

\usepackage{tikz}
\tikzset{anchorbase/.style={baseline={([yshift=-0.5ex]current bounding box.center)}}}
\usetikzlibrary{decorations.markings}
\usetikzlibrary{decorations.pathreplacing}
\usetikzlibrary{arrows,shapes,positioning,backgrounds}
\tikzstyle directed=[postaction={decorate,decoration={markings,
    mark=at position #1 with {\arrow{>}}}}]
\tikzstyle rdirected=[postaction={decorate,decoration={markings,
    mark=at position #1 with {\arrow{<}}}}]
    
\usepackage{graphicx}
\usepackage[vcentermath]{youngtab}
\usepackage{nicefrac}
\numberwithin{equation}{section}

\newtheorem{theorem}[subsubsection]{Theorem}
\newtheorem{lemma}[theorem]{Lemma}
\newtheorem{prop}[theorem]{Proposition}
\newtheorem{corollary}[subsubsection]{Corollary}

\theoremstyle{definition}
\newtheorem{definition}[subsubsection]{Definition}

\newtheorem{remark}[theorem]{Remark}

\newtheorem{example}[subsubsection]{Example}

\newtheorem{question}[theorem]{Question}

\newcommand{\bB}{\mathbf{B}}

\newcommand{\cL}{\mathcal{L}}
\newcommand{\cH}{\mathcal{H}}

\newcommand{\Fais}{\mathsf{Fais}}
\newcommand{\Sch}{\mathsf{Sch}}
\newcommand{\Fun}{\mathsf{Fun}}
\newcommand{\Alg}{\mathsf{Alg}}
\newcommand{\Set}{\mathsf{Set}}

\newcommand{\Rep}{\mathsf{Rep}}

\newcommand{\RD}{\mZ\mbox{-}\mathcal{R}\mathcal{D}}
\newcommand{\RG}{\mathcal{R}e\mathcal{G}r}
\newcommand{\pRD}{\mZ[1/p]\mbox{-}\mathcal{R}\mathcal{D}}
\newcommand{\pRG}{\mathcal{P}e\mathcal{R}e\mathcal{G}r}

\newcommand{\tto}{\twoheadrightarrow}
\newcommand{\cO}{\mathcal{O}}

\newcommand{\mN}{\mathbb{N}}
\newcommand{\mR}{\mathbb{R}}
\newcommand{\mQ}{\mathbb{Q}}
\newcommand{\mG}{\mathbb{G}}
\newcommand{\bT}{\mathbf{T}}
\newcommand{\bG}{\mathbf{G}}
\newcommand{\bH}{\mathbf{H}}
\newcommand{\bU}{\mathbf{U}}
\newcommand{\bQ}{\mathbf{Q}}
\newcommand{\mZ}{\mathbb{Z}}
\newcommand{\mA}{\mathbb{A}}
\newcommand{\sX}{\mathfrak{X}}

\newcommand{\sT}{\mathfrak{T}}
\newcommand{\sY}{\mathfrak{Y}}
\newcommand{\mC}{\mathbb{C}}
\newcommand{\mF}{\mathbb{F}}

\newcommand{\End}{\mathrm{End}}
\newcommand{\Aut}{\mathrm{Aut}}
\newcommand{\im}{\mathrm{im}}

\newcommand{\Ext}{\mathrm{Ext}}
\newcommand{\Res}{\mathrm{Res}}
\newcommand{\Hom}{\mathrm{Hom}}
\newcommand{\soc}{\mathrm{soc}}
\newcommand{\Fr}{\mathrm{Fr}}

\newcommand{\Lie}{\mathrm{Lie}}

\newcommand{\op}{\mathrm{op}}
\newcommand{\Ind}{\mathrm{Ind}}

\newcommand{\Spec}{\mathrm{Spec}}
\newcommand{\Vecc}{\mathrm{Vec}}
\newcommand{\perf}{\mathrm{perf}}
\newcommand{\red}{\mathrm{red}}

\newcommand{\id}{\mathrm{id}}

\newcommand{\blam}{{\bm{\lambda}}}
\newcommand{\bnab}{{\bm{\nabla}}}
\newcommand{\bmu}{\bm{\mu}}
\newcommand{\bnu}{\bm{\nu}}

\newcommand{\bL}{\mathbf{L}}
\newcommand{\bR}{\mathbf{R}}
\newcommand{\bX}{\mathbf{X}}
\newcommand{\bY}{\mathbf{Y}}
\newcommand{\bW}{\mathbf{W}}

\begin{document}
\title[Perfection ]{Perfecting group schemes}
\author{Kevin Coulembier and Geordie Williamson}
\address{K.C.: School of Mathematics and Statistics, University of Sydney, NSW 2006, Australia}
\email{kevin.coulembier@sydney.edu.au}

\address{G.W.: School of Mathematics and Statistics, University of Sydney, NSW 2006, Australia}
\email{g.williamson@sydney.edu.au}



\begin{abstract}
We initiate a systematic study of the perfection of affine group schemes of finite type over fields of positive characteristic. The main result intrinsically characterises and classifies the perfections of reductive groups, and obtains a bijection with the set of classifying spaces of compact connected Lie groups topologically localised away from the characteristic. We also study the representations of perfectly reductive groups. We establish a highest weight classification of simple modules, the decomposition into blocks, and relate extension groups to those of the underlying abstract group.

\end{abstract}

\maketitle


\section*{Introduction}

For a (group) scheme over a field $k$ of characteristic $p>0$, its
`perfection' is defined as the inverse limit over the Frobenius
homomorphism. In this paper we study the perfection of group schemes
and their representation theory. We place particular emphasis on
reductive groups. We obtain an intrinsic characterisation (`perfectly
reductive groups') and give a classification in terms of root data
`with $p$ inverted'. We also  give a highest weight classification of
simple modules for perfectly reductive groups, establish the block
decomposition, and make a first step towards the study of multiplicity
questions. Finally, we prove that perfectly reductive groups and the classifying
spaces of compact Lie groups localised away from $p$ are classified by
the same data. This result is the `perfect analogue' of the fact that
reductive groups over algebraically closed fields and compact Lie
groups are both classified by root data.

The motivations for the current work were three-fold, and came from
several directions. Before describing the structure of this paper in more
detail, we outline these motivations and possible future
directions.

\subsection*{Perfect representability} Sometimes
functors on rings in characteristic $p$ are only well-behaved on
perfect algebras (that is, algebras for which the Frobenius
homomorphism is an isomorphism). A prominent example is the functor of
Witt vectors and its relatives. An important observation (see for instance
\cite{BS, BD, Zh}) is that if a functor on the category of perfect
commutative $k$-algebras can be represented by a scheme, it determines
the scheme `up to perfection'. An example of such a setting is the
Witt vector affine Grassmannian, which plays a prominent role in
recent advances in the Langlands program.

It is important that passage to the perfection gives an
isomorphism of \'etale topoi. In particular, constructions built via
\'etale sheaves (like \'etale cohomology, or categories of perverse
sheaves in the \'etale topology) are insensitive to passage to the perfection. This fact
plays an important role in \cite{Zh}, where a mixed characteristic
analogue of the geometric Satake equivalence is obtained. Similarly,
it plays an important role in \cite{BD}, where the `Serre dual' of a
unipotent group is shown to be the perfection of a unipotent group,
and its character sheaves are studied.

 By the above, also the \'etale homotopy type of a (simplicial) scheme only depends
 on its perfection. In \cite{Fr}, Friedlander used this homotopy type to construct interesting maps
 between topological localisations of classifying spaces of compact Lie groups, based on (exceptional) isogenies in positive characteristic. This is one of the
main results we rely on to establish our bijection between perfectly reductive groups
and localised classifying spaces.


\subsection*{Fractal representation theory}
For a (reduced) group scheme defined over $\mathbb{F}_p$, the Frobenius twist
realises its category of representations as a full subcategory of
itself. This self-similarity induces a
fractal-like structure.
For example,
Figure~\ref{sl2pic}, shows a classic picture of the non-zero weight
spaces of simple modules for $SL_2$ in characteristic $3$. (For the
reader unfamiliar with this picture, it may be helpful to note that it
simply depicts the (non)-vanishing behaviour of Pascal's triangle
modulo $p$, see for instance \cite[\S 1]{Wi}.) This picture is fractal
like, but not genuinely fractal: one can `zoom out' but one 
cannot `zoom in' indefinitely since the Frobenius homomorphism is not
an isomorphism. By passing to the perfection, one gets a genuine
fractal. One dream (not realised in the current paper) is to use this 
fractal structure to say something about important open questions in
representation theory like dimensions and characters of simple
modules.

Much of the difficulty in the
representation theory of reductive groups in characteristic $p$ remains after perfecting. We
do observe two interesting simplifications. Firstly, the
complexities of the block decomposition disappear after passage to
the perfection (see Theorem~\ref{ThmBlock}). Secondly, perfect representation theory
appears to provide the correct setting for { the generic
  cohomology} results of \cite{CPSV}: for
a perfectly reductive group, extensions computed inside algebraic
representations agree with extensions computed as abstract groups (see
Theorem~\ref{ThmGen}). { (That perfect group schemes
  provide the correct setting for generic cohomology was suggested by
  Donkin in the early 1980s, and proved by Wang \cite{Wang}.)}

\begin{figure}[H]
\input{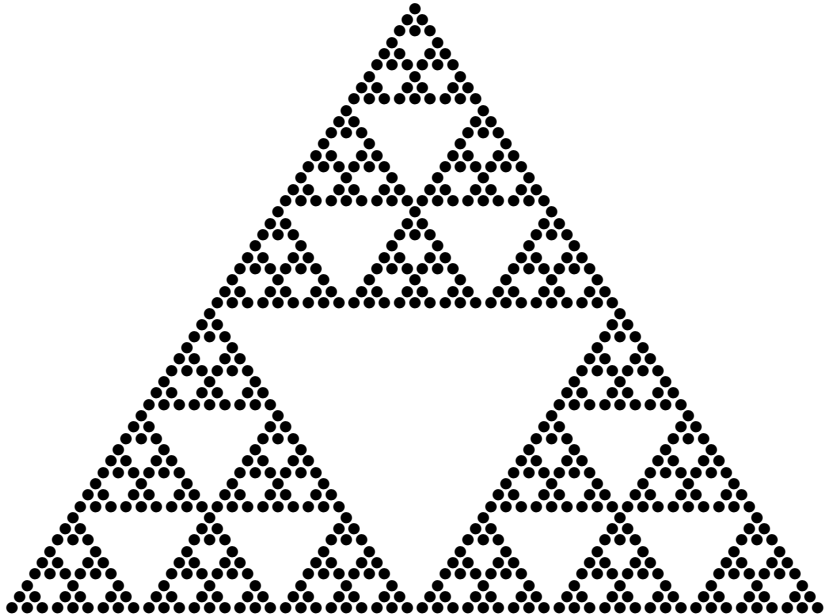}
\caption{Characters of simple modules for $SL_2$ in characteristic $p = 3$.}
\label{sl2pic}
\end{figure}

\subsection*{Tensor categories in characteristic $p$} Over fields of
characteristic zero, a famous theorem of Deligne classifies those
tensor categories which admit a fibre functor to super vector spaces
as precisely those of moderate growth \cite{Del}. It is a fascinating
open problem to find an analogue of this {  theorem} in characteristic
$p$, with many potential applications to modular representation
theory. Recently, this problem was solved in \cite{CEO} for tensor
categories with exact Frobenius functor. An important technical tool
was a limit procedure in \cite[\S 6]{CEO} which, by restriction to
representation (tensor) categories of affine group schemes,
generalises perfection of group schemes. Remarkably, the `perfection' of
a Frobenius exact tensor category of moderate growth essentially
returns the representation category of a perfect group scheme. In
other words, up to perfection, all Frobenius exact tensor categories
of moderate growth arise from (perfect) group schemes. We consider this as further evidence for
their importance.

\subsection*{Structure of the paper}
Motivated by the above considerations, we initiate a systematic study of perfecting group schemes. The paper is organised as follows:

In Section~\ref{SecData} we investigate the purely combinatorial problem of describing root data over the ring $\mZ[1/p]$, for later use in our classification results. In Section~\ref{GenPerf}, we derive some general results on perfect schemes. In Section~\ref{PerfGrp} we start studying the perfections of group schemes. We study perfect subgroups of perfect groups and their quotients. We also obtain criteria for when two group schemes perfect to isomorphic groups and derive some results on the perfections of the additive and multiplicative groups. In Section~\ref{SecRed} we classify perfectly reductive groups by their $\mZ[1/p]$-root data. In Section~\ref{SecRep} we study the representation theory of perfectly reductive groups. We classify simple modules and realise them as socles of induced modules from Borel subgroups. Then we show that the block decomposition simplifies considerably compared to the non-perfect case; in fact blocks are governed by the root lattice. We also show that extension groups for the perfected groups are given in terms of generic cohomology in the sense of \cite{CPSV}. This actually implies that extensions in the category of (rational) representations over the perfected reductive group can be computed in the category of representations of the abstract group of $\overline{\mF}_p$-points. In Section~\ref{SecBG} we prove that $\mZ[1/p]$-root data also classify the localisations away from $p$ of the classifying spaces of compact connected Lie groups. Finally, in Section~\ref{SecSL2}, we present some explicit computations for extension groups, decomposition multiplicities and line bundle cohomology for perfected $SL_2$. We also make explicit the fractal behaviour of perfected representation theory for $SL_2$.

\section{Root data over rings}\label{SecData}
For the entire section, we assume that $D$ is a {\em principal ideal
  domain of characteristic $0$}. By a $D$-lattice we understand
a finitely generated free $D$-module. Because $D$ is a PID we
  could replace `free' by `projective', so our definition agrees
  with standard terminology, e.g. in \cite{CR}. For a lattice $V$, we have the dual lattice $V^\ast:=\Hom_D(V,D)$.

\subsection{Reflection groups and root data}

We follow the definition of root data of for instance \cite{Gr}.
\begin{definition}\label{DefR}
\begin{enumerate}
\item[(0)] A {\bf reflection} $\sigma\in \Aut_D(V)$, for a $D$-lattice $V$, is a non-trivial automorphism that fixes {every element of} a submodule $V'\subset V$ for which $V/V'$ is free of rank 1.
\item A {\bf $D$-reflection group} is a pair $(W,V)$, where $V$ is a $D$-lattice and $W<\Aut_D(V)$ is a subgroup generated by reflections. We say $(W,V)$ is {\bf finite} if $W$ is a finite group.
\item A {\bf $D$-root datum} is a triple $(W,V,\{P_\sigma\})$ where $(W,V)$ is a finite $D$-reflection group and $\{P_\sigma\}$ is a collection of rank one submodules of $V$, indexed by the set $\{\sigma\}$ of reflections in $W$, satisfying 
\begin{enumerate}
\item $\im(1-\sigma)\subset P_\sigma$ and
\item $w(P_\sigma)=P_{w\sigma w^{-1}}$ for all $w\in W$.
\end{enumerate}
\end{enumerate}
\end{definition}

An isomorphism of $D$-reflection groups $(W,V)\xrightarrow{\sim}(W',V')$ is an isomorphism $\varphi:V\to V'$ with $W'=\varphi W\varphi^{-1}$.
An isomorphism of $D$-root data $(W,V,\{P_\sigma\})\xrightarrow{\sim}(W',V',\{P'_\sigma\})$ is such an isomorphism $\varphi:V\to V'$ satisfying {additionally} $\varphi(P_\sigma) = P'_{\varphi \sigma \varphi^{-1}}$ for each reflection $\sigma\in W$.

\subsubsection{}\label{DefExt1} 
A $\mZ$-reflection group $(W,V)$ yields a $D$-reflection group $D\otimes(W,V):= (W,D\otimes V)$ via extension of scalars.
Similarly, for a root datum $(W,V,\{P_\sigma\})$ over $\mZ$, we have the $D$-root datum $D\otimes (W,V,\{P_\sigma\}):=(W,D\otimes V,\{DP_\sigma\})$.

\begin{lemma}\label{LemQ}
If there exists an embedding $D\hookrightarrow \mQ$, then every $D$-root datum is the extension of scalars of a $\mZ$-root datum.
\end{lemma}
\begin{proof} 
Let $(W,V)$ be a $D$-reflection group. Starting from a $\mZ$-lattice in $V$ and acting on it with $W$ shows there exists a finitely generated $\mZ W $-submodule $V^0\subset V$ with $D\otimes_{\mZ}V^0\to V$ an isomorphism, see \cite[Corollary~23.14]{CR}. 
For a $D$-root datum $(W,V,\{P_\sigma\})$, we can then take the $\mZ$-root datum $(W,V^0,\{P^0_\sigma\})$, with $P^0_\sigma:=P_\sigma\cap V^0$.
\end{proof}

\begin{remark}
  ${}$
\begin{enumerate}
\item
  Lemma~\ref{LemQ} does not imply that root data over $D\subset\mQ$ are `the same' as root data over $\mZ$, see Example~\ref{ExampRoot}.
 \item The condition $D\subset\mQ$ is necessary in Lemma~\ref{LemQ}, as one observes by considering dihedral groups as real reflection groups on $\mR^2$, or the complex reflection groups generated by a root of unity acting by multiplication on $\mC$.
 \end{enumerate}
\end{remark}

We will use several times the following direct computations.

\begin{lemma}\label{LemComputation}
Consider a $D$-lattice $V$.
\begin{enumerate}
\item For a fixed $\phi\in V^\ast$ and $\kappa_1,\kappa_2\in V$ with $\phi(\kappa_1)=2=\phi(\kappa_2)$, we have the reflections $s_i:\lambda\mapsto\lambda- \phi(\lambda)\kappa_i$ of order $2$ on $V$. Then $$(s_1s_2)^j (\kappa_1)=\kappa_1+2j(\kappa_1-\kappa_2 ),\qquad\mbox{for all $j\in\mN$}.$$
\item { For a fixed $\kappa\in V$ and $\phi_1,\phi_2\in V^\ast$ with $\phi_1(\kappa)=2=\phi_2(\kappa)$, we have the reflections $s_i:\lambda\mapsto\lambda- \phi_i(\lambda)\kappa$ of order $2$ on $V$. Then, for all $\lambda\in V$,
$$(s_1s_2)^j (\lambda)=\lambda+j(\phi_2(\lambda)-\phi_1(\lambda))\kappa,\qquad\mbox{for all $j\in\mN$}.$$}
\end{enumerate} 
\end{lemma}

\begin{lemma}\label{Lemp2}
Consider a prime $p$ and two $\mZ$-root data $(W,V,\{P_\sigma\})$ and $(W',V',\{P_{\sigma'}'\})$. Assume there exists an isomorphism $\varphi:\mZ[1/p]\otimes(W,V)\xrightarrow{\sim}\mZ[1/p]\otimes (W',V')$ of $\mZ[1/p]$-reflection groups.
\begin{enumerate}
\item If $p=2$, then $\varphi$ is actually an isomorphism of $\mZ[1/p]$-root data.
\item If $p>2$ and if the further extension along $\mZ[1/p]\to\mZ_2=\varprojlim \mZ/{2^n}$ of $\varphi$ induces an isomorphism of $\mZ_2$-root data, then $\varphi$ is an isomorphism of $\mZ[1/p]$-root data.
\end{enumerate} 
\end{lemma}
\begin{proof}
It is well known and easy to show that for $\mZ$-root data, we either have 
\begin{equation}\label{2P}P_\sigma=\im(1-\sigma)\quad\mbox{or}\quad \im(1-\sigma)=2P_\sigma.
\end{equation} 
Hence the additional condition in the definition of an isomorphism of $\mZ[1/2]$-root data is trivially satisfied.

Now assume $p>2$. In this case, \eqref{2P} shows that in $\mZ[1/p]\otimes V'$ we have either $\varphi(P_\sigma)=P'_{\varphi\sigma \varphi^{-1}}$,  $\varphi(P_\sigma)=2P'_{\varphi\sigma \varphi^{-1}}$, or  $2\varphi(P_\sigma)=P'_{\varphi\sigma \varphi^{-1}}$. By assumption, after extension to scalars to $\mZ_2$, only the first option is possible. As $2$ is not invertible in $\mZ_2$, this means that also over $\mZ[1/p]$ only the first option was possible.
\end{proof}

\subsection{Real-type root data}

The following definition is closer to the classical definition of (reduced) root data. 

\begin{definition}\label{DefXY}
A {\bf real-type $D$-root datum} is a quadruple $(\bX,\bR,\bY,\bR^\vee)$, where $\bX$ and $\bY$ are $D$-lattices with subsets $\bR\subset \bX$ and $\bR^\vee\subset \bY$, together with
\begin{enumerate}
\item[(a)] a perfect bilinear pairing $\langle\cdot,\cdot\rangle: \bX\times\bY\to D$;
\item[(b)] a bijection $\bR\to\bR^\vee$, $\beta\mapsto \beta^\vee$;
\end{enumerate}
such that:
\begin{enumerate}
\item We have $\langle \alpha,\alpha^\vee\rangle=2$ for all $\alpha\in \bR$.
\item If $\alpha^\vee\in \bR^\vee$ and $a\in D$, then $a\alpha^\vee\in \bR^\vee$ if and only if $a\in D^\times$.
\item There are only finitely many $D^\times$-orbits in $\bR^\vee$.
\item For each $\alpha\in \bR$, the reflection $s_\alpha:\lambda\mapsto \lambda-\langle\alpha, \lambda\rangle \alpha^\vee$ in $\Aut_{D}(\bY)$ preserves $\bR^\vee$.
\item[(4')] For each $\alpha\in \bR$, the reflection $s_\alpha:\lambda\mapsto \lambda-\langle \lambda,\alpha^\vee\rangle \alpha$ in $\Aut_{D}(\bX)$ preserves $\bR$.
\end{enumerate}
\end{definition}

 {Note that $\langle s_\alpha(\lambda),s_\alpha(\mu)\rangle=\langle \lambda,\mu\rangle$, for $\lambda\in \bX$ and $\mu\in\bY$, with $s_\alpha$ as defined in (4) and~(4').}

To a real-type root datum $(X,R,Y,R^\vee)$ over $\mZ$ we can define a real-type $D$-root datum $(D\otimes X,D^\times R,D\otimes Y,D^\times R^\vee)$, with obvious bilinear pairing and bijection $D^\times R\to D^\times R^\vee$ given by $a\lambda\mapsto a^{-1}\lambda^\vee$.

\begin{remark}\label{remprime}
For $D=\mZ$, Definition~\ref{DefXY} is equivalent to the definition of a `donn\'ee radicielle r\'eduite' in \cite[\S 3.6]{De}.
We will show below {in Lemma \ref{LemBeta1}} that a real-type $D$-root datum $(\bX,\bR,\bY,\bR^\vee)$ also satisfies:
\begin{enumerate}
\item[(2')] If $\alpha\in \bR$ and $a\in D$, then $a\alpha\in \bR$ if and only if $a\in D^\times$.
\item[(3')] There are only finitely many $D^\times$-orbits in $\bR$.
\end{enumerate}
In particular, the definition of real-type root data is
closed under duality.
\end{remark}

\begin{example}\label{ExampleData}
Consider a $D$-root datum $(W,V,\{P_\sigma\})$. Take a reflection $\sigma\in W$ and a generator $v\in P_\sigma$. By condition \ref{DefR}(2)(a), there exists (a unique) $\beta\in V^\ast$ such that
\begin{equation}\label{sigmabeta}\sigma(\lambda)\;=\; \lambda-\beta(\lambda) v,\quad\mbox{for all $\lambda\in V$}.\end{equation}
We then define $\bR^\vee\subset V$ as the set of generators of the submodules $\{P_\sigma\}$ and $\bR\subset V^\ast$ as the set of elements $\beta$ constructed by the above procedure. {  If we denote by $S\subset W$ the set of reflections, then this procedure yields a surjective partially defined function
\begin{equation}\label{partfun}
S\times \bR^\vee\; \rightharpoonup \; \bR,\quad (\sigma, v)\mapsto \beta.
\end{equation}}
 \end{example}

\begin{theorem}\label{ThmData0}
Assume that for the $D$-root datum $(W,V,\{P_\sigma\})$, every reflection in $W$ has order $2$. { Then \eqref{partfun} restricts to a bijection $\bR^\vee\to\bR$.}
The quadruple $(V^\ast, \bR,V,\bR^\vee)$ equipped with inverse bijection $\bR\to\bR^\vee$  and evaluation pairing $V^\ast\times V\to D$, is a real-type $D$-root datum.
\end{theorem}

\begin{proof}
{ 
The assumption $\sigma^2=1$ in \eqref{sigmabeta} shows that $\beta(v)=2$, for every pair $(v,\beta)$ we associate to a reflection in $W$ via the procedure in \ref{ExampleData}. First we observe that this procedure yields a well-defined function $\bR^\vee\to\bR$, namely that $\beta$ only depends on $v$ and not on~$\sigma$. This is indeed the case since a given $v\in V$ cannot be a generator in both $P_{\sigma_1}$ and $P_{\sigma_2}$ for two distinct reflections $\sigma_1,\sigma_2\in W$, which follows from finiteness of $W$, Lemma~\ref{LemComputation}(2) and the fact that $D$ is of
  characteristic 0.
By definition of $\bR$, the function $\bR^\vee\to\bR$ is surjective. Next we prove that this function is injective, and hence a bijection. Assume therefore that for two generators $v_1\in P_{\sigma_1}$ and $v_2\in P_{\sigma_2}$, we obtain the same $\beta\in V^\ast$.
By finiteness of $W$ and
  Lemma~\ref{LemComputation}(1) we find $v_1=v_2$.}

By the above paragraph,  \ref{DefXY}(1) is satisfied.
To establish \ref{DefXY}(2) it suffices to show that we cannot have non-trivial inclusions $P_{\sigma_1}\subset P_{\sigma_2}$. We can extend scalars along $D\hookrightarrow K$, with $K$ the field of fractions. Now $\sigma_1\sigma_2$ acts as the identity on $KP_{\sigma_1}=KP_{\sigma_2}$, but also as the identity on $KV/KP_{\sigma_1}$. Since it has finite order and char$K=0$, we find $\sigma_1\sigma_2=1$.
Property \ref{DefXY}(3) follows immediately from the fact that $W$ is finite.
Property \ref{DefXY}(4) is an immediate consequence of \ref{DefR}(2)(b).

By letting $w\in W$ act on $V^\ast$ by $w(f)=f\circ w^{-1}$, we also have a reflection group $(W,V^\ast)$. We can define $Q_\sigma:=D\beta\subset V^\ast$, for each reflection $\sigma\in W$ with $\beta$ as in \eqref{sigmabeta}, since $Q_\sigma$ does not depend on our choice of generator $v\in P_\sigma$. It follows immediately that $(W,V^\ast, \{Q_\sigma\})$ is a $D$-root datum. That \ref{DefXY}(4') is satisfied follows by applying the proof for (4) to $(W,V^\ast, \{Q_\sigma\})$.
\end{proof}

We conclude this section with some technical results needed later.

\begin{lemma}\label{LemBeta1}
Consider a real-type $D$-root datum $(\bX,\bR,\bY,\bR^\vee)$. For $\beta\in\bR$ and $b\in D$, we have $b\beta\in \bR$ if and only if $b\in D^\times$ and then $(b\beta)^\vee=b^{-1}\beta^\vee$.
In particular, $s_{b\beta}=s_\beta$ and conditions (2') and (3') in \ref{remprime} hold.
\end{lemma}
\begin{proof}
 {Assume first that $b\in D^\times$. We need to prove that $(\sigma, b^{-1}\beta^\vee)\mapsto b\beta$ in \eqref{partfun}, for $\sigma:\lambda\mapsto \lambda-\beta(\lambda)\beta^\vee$, which is clearly true.}

Conversely, assume that $\beta_1:=b\beta\in\bR$, for some $b\in D$, and set $\lambda:=b\beta_1^\vee-\beta^\vee\in \bY$. Lemma~\ref{LemComputation}(1) for $\phi=\langle\beta,-\rangle$ and $\kappa_1=b\beta_1^\vee$, $\kappa_2=\beta^\vee$ implies
$$b\,(s_{\beta_1}s_\beta)^j(\beta_1^\vee)\;=\;b\beta_1^\vee+2j\lambda,\quad\mbox{for all $j\in\mN$}.$$
Conditions~\ref{DefXY}(3) and (4) thus imply $\lambda=0$. Finally \ref{DefXY}(2) then implies $b$ is a unit.
\end{proof}

\begin{lemma}\label{LemABC}
Consider a real-type $D$-root datum $(\bX,\bR,\bY,\bR^\vee)$. If, for $\alpha,\beta,\gamma\in\bR$, we have $s_\beta(\alpha^\vee)=\gamma^\vee$, then $s_\beta s_\alpha s_\beta=s_\gamma$ and $s_\beta(\alpha)=\gamma$.
\end{lemma}
\begin{proof}
A direct calculation shows that, for every $\lambda\in\bY$,
\begin{equation}\label{sss}
s_\beta s_\alpha s_\beta(\lambda)\;=\; \lambda-\langle s_\beta(\alpha),\lambda\rangle s_\beta(\alpha^\vee).
\end{equation}
Set $\gamma_1= s_\beta(\alpha)\in\bR$. Then we can apply Lemma~\ref{LemComputation}(1) to $\phi=\langle\gamma_1,-\rangle$ and $\kappa_1=\gamma_1^\vee$, $\kappa_2 =\gamma^\vee$ (so $s_1=s_{\gamma_1}$ and $s_2=s_\beta s_\alpha s_\beta$). By \ref{DefXY}(3) and (4), it thus follows that $\gamma_1=\gamma$.
\end{proof}

\subsection{Real reflection groups}

\subsubsection{Hypotheses}\label{hypo}
Consider a finite-dimensional real vector space $V$ (without fixed inner product/Euclidean structure), a reflection group $W<\Aut_{\mR}(V)$, in the sense of Definition~\ref{DefR}(1), and a fixed {\em finite} generating set $T$ of reflections in $W$ such that
\begin{itemize}
\item The map from $T$ to the set of hyperplanes in $V$, $s\mapsto H_s:=\ker (1-s)$, is injective.
\item We have $wTw^{-1}\subset T$ for all $w \in W$.
\end{itemize}

 \begin{theorem}\label{ThmWFin}
 Under the assumptions in \ref{hypo}, $W$ is a finite group
 \end{theorem}
 
 \begin{remark}\label{RemEuc}
 If a real reflection group $(W,V)$ is finite, by Weyl's unitary trick, we can assume it is Euclidean (meaning there is an inner product on $V$ for which each reflection in~$W$ is orthogonal).
 \end{remark}
 The remainder of this section is devoted to the proof.

 \subsubsection{}

All topological references consider the Euclidean topology on $V$. Consider
$$\cH=\cup_{t\in T}H_t\,\subset\, V$$
and refer to the connected components of the complement of $\cH$ in $V$ as
chambers.  {We say that $H_t$ is a {\bf wall} of a chamber $A$ if the intersection of $\overline{A_0}$ and $H_t$ cannot be contained in a codimension 2 hyperplane.}
Fix one such chamber $A_0$.
Denote by $S\subset T$ the set of reflections $s$ for which $H_s$ is a wall of $A_0$.
Our assumptions in \ref{hypo} imply that $W$ acts on the (finite) set of chambers.

\begin{lemma}
The set $S$ is a set of generators for $W$.
\end{lemma}
\begin{proof}
 {Denote by $W_S\subset W$ the subgroup generated by $S\subset T$. We need to show that $W_S=W$, or equivalently $T\subset W_S$.}

First we show that every  {$W_S$-orbit} in $V$ intersects $\overline{A_0}$. By continuity it suffices to show that every $W_S$-orbit in $V\backslash \cH$ intersects $A_0$. For $v\in V\backslash \cH$, there exists a sequence $A_0,A_1,\cdots, A_l$ of distinct chambers where $v\in A_l$, and for each $0\le i<l$ there is $t_i\in T$ such that $H_{t_i}$ is a wall of $\overline{A_i}$ and of $\overline{A_{i+1}}$.

If $l=0$ there is nothing to prove, so assume $l>0$. Then $t_0(A_1)=A_0$, and by assumption $t_0\in S$. 
Now $A_0, A_1'=t_0(A_2),A_2'=t_0(A_3),\cdots A_{l-1}'= t_0(A_l)$ forms a chain of distinct chambers as before and $t_0(v)\in A'_{l-1}$. We can thus perform induction on $l$ to deduce the claim.

 {Now take an arbitrary $t\in T$ and let $A$ be a chamber which has $H_t$ as a wall. By the above, we know there exists $w\in W_S$ with $A=w(A_0)$. This implies that
$$H_t\;=\; w(H_s)\;=\;H_{wsw^{-1}}.$$
By the first hypothesis in \ref{hypo}, this shows that $t=wsw^{-1}\in W_S$.}
\end{proof}


Precisely as in \cite[V.\S 3.2 Theorem~1]{Bo}, we then find the following consequence.
\begin{corollary}\label{CorCox}
The pair $(W,S)$ is a Coxeter system. 
\end{corollary}

\begin{proof}[Proof of Theorem~\ref{ThmWFin}]
The reflections in the Coxeter group $(W,S)$ are by definition the elements of the set $\cup_{w\in W} wSw^{-1}$, which by assumption is included in $T$ and hence finite.
By \cite[Corollary 1.4.5]{BB} any Coxeter group $(W,S)$ with finitely many reflections is finite.
\end{proof}

\subsection{Equivalence of definitions}

\begin{theorem}\label{ThmData}
If there exists an embedding $D\hookrightarrow \mR$, then the map in Theorem~\ref{ThmData0} is a bijection between the sets of isomorphism classes of $D$-root data and real-type $D$-root data.\end{theorem}
\begin{proof}
Since we have $D\subset \mR$, the only roots of unity in $D$ are $\pm 1$ and it follows that every reflection of finite order must have order two. Hence the map in Theorem~\ref{ThmData0} is defined on every $D$-root datum.

To each real-type $D$-root datum $(\bX,\bR,\bY,\bR^\vee)$ we will now associate a $D$-root datum, in a way which is easily seen to be the inverse of the above map.
Define the $D$-reflection group $W<\Aut_{D}(\bY)$ generated by $\{s_\alpha\,|\, \alpha\in\bR\}$. To a reflection $s_\gamma$, $\gamma\in\bR$, we associate the corresponding rank one submodule $D\gamma^\vee\subset \bY$.

 We show that $W$ is finite by considering the corresponding real reflection group acting on $\bY\otimes_{D}\mR$. By Lemma~\ref{LemBeta1}, $W$ is generated by a finite (see \ref{DefXY}(2)) set of reflections $\{s_\alpha\,|\, \alpha \in\bR\}$, such that the reflecting hyperplane $\ker (1-s_\alpha)=\ker\langle\alpha,- \rangle$ determines $s_\alpha$. Moreover, we claim that for each $\alpha\in \bR$ and $w\in W$, we have $ws_\alpha w^{-1}= s_{\gamma}$ for some $\gamma\in \bR$. Clearly it suffices to consider the case $w=s_\beta$, which is Lemma~\ref{LemABC}.
 We can now apply Theorem~\ref{ThmWFin}.

Now it follows that the triple $(W,\bY,\{D\gamma^\vee\})$ is a $D$-root datum. Indeed, by Remark~\ref{RemEuc} and \cite[Proposition~1.14]{Hu}, every reflection in $W$ is equal to $s_\gamma$ for some $\gamma\in \bR$, and property (a) in \ref{DefR}(2) is automatic, while (b) follows from Lemma~\ref{LemABC}. 
\end{proof}

Clearly, the bijection in Theorem~\ref{ThmData} exchanges the two notions of extensions of scalars of root data. We conclude this section with some examples of root data which become isomorphic after extension of scalars.
\begin{example}\label{ExampRoot}${}$
\begin{enumerate}
\item The root datum of $SO_{2n+1}$ becomes isomorphic to its dual (the root datum of $Sp_{2n}$) after extension of scalars to $D$ if and only if $2$ is invertible in $D$.
\item The root datum of $SL_n$ becomes isomorphic to its dual (the root datum of $PGL_n$) after extension of scalars to $D$ if and only if $n$ is invertible in $D$.
\end{enumerate}
\end{example}

%
%

\subsection{Isogenies}\label{SecIso1}
We fix a prime $p$ and define isogenies of $\mZ[1/p]$-root data. We
will work with real-type root data, but refer to them simply as root
data (this is justified by Theorem~\ref{ThmData}).

\begin{definition} \label{definition_isogeny}
An isogeny $(\bX,\bR,\bY,\bR^\vee)\to (\bX_1,\bR_1,\bY_1,\bR_1^\vee)$ of $\mZ[1/p]$-root data is an injective morphism $\varphi:\bX\hookrightarrow \bX_1$ of $\mZ[1/p]$-modules 
such that
\begin{enumerate}
\item the induced $\varphi^\vee:\bY_1\to \bY$ is also injective;
\item $\varphi$ restricts to a bijection $\bR\to\bR_1$;
\item $\varphi^\vee(\varphi(\alpha)^\vee)=\alpha^\vee$, for all $\alpha\in\bR$.
\end{enumerate}
\end{definition}

\begin{example}\label{RemIsog}
An isogeny of $\mZ$-root data, with respect to some prime $p$, is
defined in \cite[\S 1]{St}. It follows immediately that the induction
to $\mZ[1/p]$ of such an isogeny yields an isogeny of $\mZ[1/p]$-root
data. Note that Definition \ref{definition_isogeny} is simpler than the definition in \cite[\S 1]{St}, as
  the powers of $p$ present in \cite[\S 1]{St} are subsumed by (2), since
  multiplication by $p$ is invertible on $\mZ[1/p]$-root data.

Conversely, if for $\mZ$-root data $RD_1$ and $RD_2$, there exists an isogeny $\varphi:\mZ[1/p]\otimes RD_1\to \mZ[1/p]\otimes RD_2$, then for some $l\in\mN$, the map $p^j\varphi$ restricts to an isogeny $RD_1\to RD_2$ in the sense of \cite{St} for all $j\ge l$.
\end{example}


\section{Perfection of schemes}\label{GenPerf}

Fix a prime $p$.

\subsection{Notation}We recall some basic set-up of algebraic geometry, see for instance~\cite{DG}.

Fix a field $k$. Denote by $\Alg_k$ the category of commutative $k$-algebras. We consider the categories (where the first two `inclusions' are fully-faithful embeddings)
\begin{equation}\label{3Inc}\Alg_k^{\op}\;\subset\; \Sch_k\;\subset\; \Fais_k\;\subset\; \Fun_k.\end{equation}
Here $\Sch_k$ is the category of $k$-schemes and $\Fun_k$ is the category of functors $\Alg_k\to\Set$. The category $\Fais_k$ stands for the full subcategory of such functors which are sheaves for the fpqc topology. In other words, a functor $F$ is in $\Fais_k$ if and only if
$$F(A)\to F(B)\rightrightarrows F(B\otimes_AB)$$
is an equaliser for every faithfully flat $A$-algebra $B$, and $F$ commutes with finite products. When $k$ is clear, we will usually leave out the subscript in the above categories.

The inclusion $I:\Fais\hookrightarrow \Fun$ has a left adjoint 
$$S:\Fun\to\Fais$$
which commutes with finite limits (as well as all colimits).

By a subgroup of an affine group scheme $G$ we understand a closed subscheme which inherits a group structure, or in other words
an affine group scheme represented by a quotient of the Hopf algebra representing $G$.

\subsection{Perfection functors}

\subsubsection{Frobenius morphisms}\label{DefFrob}
{ For a commutative $\mF_p$-algebra $A$, we have the $p$-th power algebra morphism 
$$\Fr=\Fr_A:\,A\to A, \;\; a\mapsto a^p.$$ For an $\mF_p$-scheme $\sX$, we have the morphism $\Fr: \sX\to\sX$ which is the identity on the underlying topological space and given by the $p$-th power map on the sheaf of algebras. For $F\in\Fun_{\mF_p}$, we define the Frobenius morphism $F\to F$ as the natural transformation given by letting $F$ act on the $p$-th power morphism. Concretely, the evaluation of the natural transformation at an arbitrary $A\in\Alg_{\mF_p}$ is
$$F(A)\xrightarrow{F(\Fr_A)} F(A).$$ These Frobenius morphisms are compatible with the inclusions \eqref{3Inc}.}

For an object $F$ of $\Fun_{\mF_p}$ or $\Alg_{\mF_p}$, the notation $\varinjlim F$ or $\varprojlim F$ will always be used for the direct or inverse limit along the Frobenius morphism. {  For example, for an $\mF_p$-algebra $A$, the algebra $A_{\perf}:=\varinjlim A$ is the direct limit of the system
$$ A\xrightarrow{a\mapsto a^p}A\xrightarrow{a\mapsto a^p}A\xrightarrow{a\mapsto a^p}A\to\cdots.$$
For $F\in \Fun_{\mF_p}$, it follows directly that
$$(\varprojlim F)(A)=F(\varprojlim A)\quad\mbox{and}\quad (\varinjlim F)(A)=F(\varinjlim A).$$}

An $\mF_p$-scheme (or an algebra or functor) is called {\bf perfect} if the Frobenius map is an isomorphism, see \cite[Definition~3.1]{BS}.

\begin{lemma} \label{lem:perf}
The endofunctor of $\Fun_{\mF_p}$
$$F\mapsto F_{\perf}:=\varprojlim F$$
restricts to endofunctors of $\Fais$ and $\Sch$. Moreover, for $A\in \Alg_{\mF_p}$, we have $(\Spec A)_{\perf}=\Spec (A_{\perf})$.
\end{lemma}
\begin{proof}
It is a standard property that limits exist in a Grothendieck topos and can be computed in the presheaf category, which shows that perfection restricts to $\Fais$.
The remaining properties follow from the explicit realisation in Example~\ref{ExDef} below.
\end{proof}

\begin{example}\label{ExDef}
For an $\mF_p$-algebra $A$, set $(X,\cO)=\Spec A$. Using the basis of distinguished open subsets it follows easily that $\Spec (A_{\perf})=(X,\cO')$, with $\cO'$ the sheafification of the presheaf $U\mapsto \cO(U)_{\perf}$.

It then follows that for an arbitrary $\mF_p$-scheme $\sX=(X,\cO)$, the scheme $\sX_{\perf}$ can be realised as $(X,\cO')$ with $\cO'$ the sheafification of the presheaf $U\mapsto \cO(U)_{\perf}$.
\end{example}

\begin{remark}
In \ref{ExDef} is essential to take $\cO'$ as the direct limit of $\cO$ in the category of sheaves (as opposed to presheaves). For example:
\begin{enumerate}
\item For an infinite family of $\mF_p$-algebras $A_i$, consider the (non-affine) scheme $(X,\cO)=\sqcup_i\Spec A_i$. Then, for general $A_i$, we have, by the sheaf axioms and Lemma~\ref{lem:perf},
$$\Gamma(X,\cO')\;=\;\prod_i (A_i)_{\perf}\;\not=\; \Gamma(X,\cO)_{\perf}=\left(\prod_i A_i\right)_{\perf}.$$
 \item Also for non-noetherian affine schemes this phenomenon occurs. Consider $A=\mF_p[x_i\,|\,i\in\mN]/(x_ix_j, i\not=j)$. Let $U_i$ be the distinguished open corresponding to $x_i$. The (disjoint) union $U=\cup_i U_i$ is the complement of the origin and, as in (1), we find $\cO(U)_{\perf}\not=\cO'(U)$.
\end{enumerate}

\end{remark}

\begin{remark}\label{remaffsch}
For $\mF_p$-algebras $A,B$ we have
$$\Alg(A_{\perf},B)\cong \varprojlim \Alg(A, B)\cong\Alg(A,\varprojlim B).$$
In particular, if $B$ is perfect, we have $\Alg(A_{\perf},B)\cong \Alg(A,B)$.
\end{remark}

\begin{lemma}\label{LemBasic} Let $\sX$ be an $\mF_p$-scheme.
\begin{enumerate}
\item We have $\dim\sX=\dim\sX_{\perf}$ and $\sX_{\perf}$ is quasi-compact (resp. connected) if and only if $\sX$ is quasi-compact (resp. connected). 
\item Any radical ideal $I$ in a perfect $\mF_p$-algebra $A$ satisfies $I^2=I$.
\item If $\sX$ is perfect, for $x\in\sX$ we have $T_{\sX,x}=0$.
\item Perfect schemes are reduced. Moreover, $-_{\perf}$ sends $\sX_{\red}\to\sX$ to an isomorphism.
\end{enumerate}
\end{lemma}
\begin{proof}
Part (1) follows immediately from the fact that the underlying topological spaces of $\sX$ and $\sX_{\perf}$ are the same, see Example~\ref{ExDef}. Part (2) is obvious. By (2), it is clear that the Zariski cotangent space is zero, which proves (3).
Alternatively, for (3), let $A$ be a perfect $\mF_p$-algebra and $\kappa$ a field. Every algebra morphism $A\to\kappa[\epsilon]/(\epsilon^2)$ factors through $\kappa\hookrightarrow\kappa[\epsilon]/(\epsilon^2)$. Applying this to $\Spec (\kappa(x)[\epsilon]/(\epsilon^2))\to\sX$ shows the claim. 
Part (4) is immediate.
\end{proof}


\subsection{Relative version}
Fix a perfect field $k$ of characteristic $p$, for the remainder of the section.

\subsubsection{}\label{relative1}
For a fixed $\mF_p$-scheme $\sT$, perfection naturally yields a functor from the category of $\sT$-schemes to the category of $\sT_{\perf}$-schemes. Using the canonical morphism $\sT_{\perf}\to\sT$, we can also interpret perfection as an endofunctor of the category of $\sT$-schemes. We will take the latter point of view for $\sT=\Spec k$ (in which case $\sT_{\perf}\to\sT$ is an isomorphism) and we will henceforth interpret the perfection functor as an endofunctor of $\Sch_k$.

\subsubsection{}\label{relative2} Sometimes it will be beneficial to consider an alternative realisation of the perfection of $k$-schemes, in which the morphisms in the chain of which we take the limit are morphisms of $k$-schemes. For a $k$-scheme $\sX$, let $\sX^{(1)}$ denote the extension of scalars of $\sX$ along the Frobenius automorphism $k\to k$. The morphism $\Fr:\sX\to \sX$ over $\mF_p$ from \ref{DefFrob} then lifts to a morphism $\Fr:\sX\to \sX^{(1)}$ of $k$-schemes.
For instance, for a $k$-algebra $A$ this corresponds to the morphism
\begin{equation}\label{eqRelFr}A^{(1)}\to A,\quad \lambda\otimes a\mapsto \lambda\,a^p,\end{equation}
with $A^{(1)}=k\otimes A$ the extension of scalars along the Frobenius automorphism of $k$.

By taking iterates of the Frobenius automorphism and its inverse ($k$ is perfect), we define $\sX^{(i)}$ for $i\in\mZ$. Then we have (over $k$)
$$\sX_{\perf}\,\cong\,\varprojlim_{i\to\infty}\sX^{(-i)}.$$

The advantage of the approach in this subsection is that it extends to $\Fun_k$, by setting
$$F_{\perf}(A):=\varprojlim_{i\to\infty} F(A^{(i)}).$$
By construction, perfection commutes with limits, for instance products, in $\Fun_k$.

\begin{prop}\label{PropEpi}
Consider a morphism $f$ in $\Fais_k$.
\begin{enumerate}
\item  If $f$ is an epimorphism in $\Fais_k$, then so is $f_{\perf}$.
\item If $f$ is an monomorphism in $\Fais_k$, then so is $f_{\perf}$.
\end{enumerate}
\end{prop}
\begin{proof}
For part (1), we can use the criterion from \cite[Corollaire~2.8]{DG} to describe that $f$ is an epimorphism, which carries over to $f_{\perf}$ by \cite[Lemma~3.4(xii)]{BS}. Part (2) 
is a generality for limits of monomorphisms.
\end{proof}

It is obvious that $\sX\mapsto \sX_{\perf}$ loses a lot of information. For instance 
$$(\sX_\perf)_{\perf}\cong \sX_{\perf}\cong (\sX_{\red})_{\perf}.$$ A more subtle example is given below.

\begin{example} Assume $p>2$.
Consider the algebra $A=k[x,y]/(y^p-x^2)$ with injective algebra morphism $A\hookrightarrow k[z]$, given by $x\mapsto z^p$, $y\mapsto z^2$. Then $\sX:=\Spec A$ is reduced, but perfection sends $\mA^1_{k}\to \sX$ to an isomorphism.
\end{example}

\subsection{Perfect finite type}

Recall that $\sX\in \Sch_k$ is of finite type (over $k$) if the underlying topological space is quasi-compact and for every $x\in \sX$ there exists an affine open neighbourhood isomorphic to the spectrum of a finitely generated $k$-algebra.

%

\begin{lemma}\label{LemFinGen}
For a perfect commutative $k$-algebra $A$, the following conditions are equivalent:
\begin{enumerate}
\item[(a)] There is a finite subset $S\subset A$ such that the set 
$$S'=\{x\in A\,|\, x^{p^n}\in S\mbox{ for some $n\in\mN$}\}$$
generates $A$ as a $k$-algebra;
\item[(b)] There exists a finitely generated $k$-algebra $A_0$ with $A\cong (A_0)_{\perf}$.
\end{enumerate}
If the conditions are satisfied we say that $A$ is {\bf perfectly finitely generated}.
\end{lemma}
\begin{proof}
Exercise.
\end{proof}

\begin{prop}
For a perfect $k$-scheme $\sX$, the following are equivalent:
\begin{enumerate}
\item[(a)] $\sX$ is quasi-compact and for every $x\in \sX$ there exists an affine open neighbourhood corresponding to a perfectly finitely generated $k$-algebra;
\item[(b)] There exists a scheme $\sY$  of finite type over $k$ with $\sX\cong \sY_{\perf}$.
\end{enumerate}
If the conditions are satisfied we say that $\sX$ is of {\bf perfect finite type} (over $k$).
\end{prop}
\begin{proof}
Clearly (b) implies (a). That (a) implies (b) is proved in \cite[Proposition~3.13]{BS}.
\end{proof}
 

\subsection{Perfection of line bundle cohomology and quotients}

All schemes and functors are assumed to be over $k$.

\begin{lemma}\label{LemLine}
Let $\sX$ be a quasi-compact separated scheme over $k$ and $\cL$ a line bundle on~$\sX$. For the pull-back $p^\ast \cL$ along $p:\sX_{\perf}\to \sX$ and $i\in \mN$, we have
$$H^i(\sX_{\perf},p^\ast \cL)\;\cong\; \varinjlim_j H^i(\sX,\cL^{\otimes p^j}),$$
where the transition maps are induced from $\cL^{\otimes p^j}\to \cL^{\otimes p^{j+1}}$, $f\mapsto f^{\otimes p}$.
\end{lemma}
\begin{proof}
Since $\sX$ is quasi-compact we can take a finite cover $\mathbb{U}$ by affine opens and, since $\sX$ is separated, intersections of these opens are again affine. Moreover, the cohomology groups $H^i(\sX,-)$ are canonically isomorphic to the \v{C}ech cohomology groups $\check{H}^i(\mathbb{U},-)$. It follows that $H^i(\sX,-)$ commutes with direct limits.

Now, for a line bundle, we have (as sheaves on the underlying topological space of $\sX_{\perf}$ or $\sX$) isomorphisms
$$p^\ast \cL\;\cong\; (\varinjlim \cO_{\sX})\otimes_{\cO_{\sX}}\cL\;\cong\; \varinjlim \cL^{\otimes p^j},$$
which follow easily from Remark~\ref{RemLine} below.

The conclusion follows from the combination of the two paragraphs.
\end{proof}

 {
\begin{remark}\label{RemLine}
In case $\sX=\Spec A$ for a commutative $k$-algebra $A$, Lemma~\ref{LemLine} simply states that
$$A_{\perf}\otimes_A\cL\;\cong\; \varinjlim \cL^{\otimes p^j},$$
for every invertible $A$-module $\cL$. To prove this isomorphism, we only need to observe that the morphism
$$A\otimes_A \cL\;\to\;\cL^{\otimes p},\quad a\otimes v\mapsto a v^{\otimes p},$$
where the first module is the extension of scalars of $\cL$ along the $p$-th power map, is an isomorphism. Indeed, the morphism commutes with localisation and is obviously an isomorphism when $A$ is local.
\end{remark}}

\subsubsection{}For $X\in \Fais$ and $G$ a group object in $\Fais$ acting on $X$ (on the right), we consider the corresponding co-equalisers in $\Fun$ and $\Fais$ of $X\times G\rightrightarrows X$, and denote them by $X/_0G$ and $X/_1 G$. In particular $X/_1G=S(X/_0G)$, for the sheafification functor $S:\Fun\to\Fais$.

It follows from the definitions that $G_{\perf}$ is again a group object, and acts on $X_{\perf}$. We create the following commutative diagram in $\Fun$
\begin{equation}\label{eqQuo}\xymatrix{X_{\perf}\ar[rr]&&X_{\perf}/_0G_{\perf}\ar[rr]\ar[d]&& (X/_0G)_{\perf}\ar[d]\\
&&X_{\perf}/_1 G_{\perf}\ar[rr]&& (X/_1G)_{\perf}.} \end{equation}
The vertical arrows are induced from the adjunction $S\dashv I$ (either directly or via the action of the perfection functor). The left horizontal arrow is the defining one for the co-equaliser. The remaining two arrows are uniquely defined from the coequaliser properties applied to the perfection of the morphisms $X/_1G\leftarrow X\rightarrow X/_0G$.


\begin{theorem}\label{ThmQuot1}
Assume that the action of $G$ on $X$ is free, then the morphism from \eqref{eqQuo} 
$$X_{\perf} /_1 G_{\perf}\;\to\; (X/_1 G)_{\perf}$$
in $\Fais_k$ is an isomorphism.
\end{theorem}
\begin{proof}
By Proposition~\ref{PropEpi}(1), the composite morphism (from top left to bottom right) in~\eqref{eqQuo} is an epimorphism in $\Fais$. In particular the lower horizontal arrow is an epimorphism. Since isomorphisms in Grothendiek topoi are precisely morphisms which are both monomorphisms and epimorphisms, it now suffices to show this arrow is also a monomorphism.

Since sheafification sends monomorphisms to monomorphisms, it actually suffices to show that $X_{\perf}/_0G_{\perf}\to  (X/_1G)_{\perf}$ is a monomorphism in $\Fun$. Using the assumption that the action is free, we can prove that $X_{\perf}/_0G_{\perf}\to (X/_0G)_{\perf}$ and $(X/_0G)_{\perf}\to(X/_1G)_{\perf}$ are monomorphisms.

Indeed, the first case follows from the fact that for an inverse system of sets $X_i$ with {\em free} actions of groups $G_i$, 
$$\varprojlim X_i/\varprojlim G_i\;\to\;\varprojlim (X_i/G_i)$$
is injective. By Proposition~\ref{PropEpi}(2), for the second morphism, it suffices to demonstrate that $X/_0 G\to X/_1G$ is a monomorphism. This is equivalent to the claim that the presheaf $X/_0 G$ is separated for the fpqc topology, meaning that 
$$X(A)/G(A)\;\to \; X(B)/G(B)$$
is an injection for every faithfully flat $A$-algebra $B$ (and the same for finite products of algebras). This is easily verified for free actions, see for instance \cite[\S I.5.5]{Jantzen}.
\end{proof}



\section{Perfection of group schemes}\label{PerfGrp} Let $k$ be a perfect field of characteristic $p>0$. All schemes are assumed to be over $k$.

\subsection{Perfection}

For an affine group scheme $G$ over $k$, clearly $G_{\perf}$ is again an affine group scheme over $k$. Note also that, since $k$ is perfect, $G_{\red}$ is a subgroup of $G$.

\subsubsection{The group $p^{\mZ}$}\label{pZ} We denote by $p^{\mZ}<\mZ[1/p]^\times$ the group of powers of $p$, an infinite cyclic group. Let $G$ be an affine group scheme over $k$, which can be defined over $\mF_p$. Then we can choose an isomorphism $\phi:G^{(1)}\xrightarrow{\sim} G$, which yields an automorphism
$$\Phi=\phi_{\perf}\circ \Fr:\;G_{\perf}\;\xrightarrow{\sim}\; G_{\perf}^{(1)}\;\xrightarrow{\sim}\; G_{\perf},$$
and a corresponding group homomorphism $p^{\mZ}\to \Aut(G_{\perf})$, $p\mapsto \Phi$. Examples are given in~\ref{ExZ1p}.

\begin{lemma}
For a perfect group scheme $G$, we have $\Lie G=0$ and $\mathrm{Dist}G=k$.
\end{lemma}
\begin{proof}
This is an immediate consequence of Lemma~\ref{LemBasic}(2) and (3).
\end{proof}

\begin{theorem}\label{PerfAlgGr}
Let $\bG$ be a perfect affine group scheme over $k$. The following are equivalent:
\begin{enumerate}
\item[(a)] The scheme $\bG$ is of perfect finite type, {\it i.e.} $k[\bG]$ is perfectly finitely generated;
\item[(b)] The group scheme $\bG$ is a subgroup of $GL(V)_{\perf}$ for a finite-dimensional vector space~$V$;
\item[(c)] There exists an affine group scheme $G$ of finite type with $\bG\cong G_{\perf}$;
\item[(d)] There exists a reduced affine group scheme $G$ of finite type with $\bG\cong G_{\perf}$.
\end{enumerate}
\end{theorem}
\begin{proof}
Clearly (d) implies (c). Any affine group scheme $G$ of finite type is a subgroup of some $GL(V)$, see e.g.~\cite[Corollary~2.5]{DM}. It follows immediately that $G_{\perf}<GL(V)_{\perf}$, so (c) implies (b). That (b) implies (a) follows by the observation that $k[\bG]$ is a quotient of the perfectly finitely generated algebra $\varinjlim k[GL(V)]$.

Finally, we prove that (a) implies (d). Let $S$ be a finite set which perfectly generates $k[\bG]$ as in Lemma~\ref{LemFinGen}(a). { It is possible to replace $S$ by a finite set $S_1\supset S$ such that the subalgebra that is generated (in the ordinary, non-perfect sense) by $S_1$ is actually a Hopf subalgebra of $k[\bG]$, see \cite[Lemma~3.4.5]{Ab}.} Let $G$ be the corresponding affine group scheme. By construction $G_{\perf}\cong \bG$ and, since $k[G]$ is a subalgebra of $k[\bG]$, it follows that~$G$ is reduced.
\end{proof}

\begin{lemma}\label{LemGHperf}
Let $\bG$ be a perfect affine group scheme and $H$ an affine group scheme. Then $H_{\perf}\to H$ induces an isomorphism
$$\Hom(\bG, H_{\perf})\;\xrightarrow{\sim}\;\Hom(\bG,H),$$
with inverse given by perfection.
Moreover, if $H$ is of finite type and $G$ an affine group scheme, then
$$\Hom(G_{\perf}, H_{\perf})\;\xrightarrow{\sim}\;\varinjlim\Hom(G^{(-i)},H).$$
\end{lemma}
\begin{proof}
This is an immediate application of Remark~\ref{remaffsch}, or the fact that the perfection functor on $\mF_p$-schemes is right adjoint to the inclusion functor for perfect schemes.
\end{proof}

\subsubsection{}
For a subgroup $H$ of an affine group scheme $G$, we denote by $G/H$, when it exists, the equaliser of $G\times H\rightrightarrows G$ in $\Sch_k$.

Recall from \cite[III, \S3 Th\'eor\`eme 5.4]{DG} that for $G$ of finite type, the quotient $G/_1H$ in $\Fais_k$ is a scheme, of finite type over $k$, so in particular is equal to $G/H$.

\begin{theorem}\label{ThmSub}
\begin{enumerate}
\item For every perfect subgroup $\bH$ of an affine group scheme $\bG$ of perfect finite type, the quotient $\bG/\bH$ exists, is of perfect finite type and isomorphic to $\bG/_1\bH$.

\item For an affine group scheme $G$, every perfect subgroup of $G_{\perf}$ is the perfection of a subgroup of $G$. More precisely, every perfect (normal) subgroup of $G_{\perf}$ is the perfection of a reduced (normal) subgroup of $G_{\red}<G$.

\item For an affine group scheme $G$ of finite type with subgroup $H$, the quotient $G_{\perf}/H_{\perf}$ exists and is isomorphic to $(G/H)_{\perf}$ and $G_{\perf}/_1H_{\perf}$.
\end{enumerate}
\end{theorem}
\begin{proof}
We will freely use the results from \cite{DG} recalled above.
Part (1) is then an immediate consequence of parts (2) and (3). 

Now we prove part (2). Take a perfect subgroup $\bH<G_{\perf}$. We have a commutative square, where $\hookrightarrow$ denotes the inclusion of a subgroup and $\tto$ denotes a faithfully flat homomorphism
$$\xymatrix{
\bH\ar@{^{(}->}[rr]\ar@{->>}[d]&& G_{\perf}\ar@{->>}[d]\\
L\ar@{^{(}->}[rr]&& G_{\red},
}$$
where $L$ is just defined to be the image of the composite diagonal homomorphism. By Lemma~\ref{LemGHperf}, perfecting the lower path in the square yields homomorphisms $\bH\tto L_{\perf}\hookrightarrow G_{\perf}$ which compose to the original inclusion $\bH\to G_{\perf}$. Clearly, $\bH\tto L_{\perf}$ must be an isomorphism. If $\bH$ is a normal subgroup, it follows
easily that so is $L<G_{\red}$.

Part (3) is an application of Theorem~\ref{ThmQuot1}.
\end{proof}

\begin{remark}
A perfect group scheme can have non-perfect subgroups, for instance 
$$\mu_{p^\infty}:=\varprojlim \mu_{p^i},\quad i.e. \;\;k[\mu_{p^\infty}]=k[x^{1/p^\infty}]/(x-1)$$ is a subgroup of $(\mG_m)_{\perf}$. {  Here we use the convention
$$k[x^{1/p^\infty}]\;=\; k[x,x^{1/p},x^{1/p^2},\cdots]=\cup_i k[x^{1/p^i}].$$} Moreover, $(\mG_m)_{\perf}/\mu_{p^\infty}\cong \mG_m$.
\end{remark}

For our purposes it is most convenient to define short exact sequences of affine group schemes as those sequences $N\to G\to Q$ in which $G\to Q$ is faithfully flat and $N$ is the kernel of the latter morphism. 
\begin{lemma}\label{CorSES}
The perfection functor acting on a short exact sequence of affine group schemes
$$1\to N\to G\to Q\to 1,$$
yields a short exact sequence
$$1\to N_{\perf}\to G_{\perf}\to Q_{\perf}\to 1.$$
\end{lemma}
\begin{proof}
Faithful flatness is preserved by perfection, see \cite[Lemma~3.4]{BS}. Taking inverse limits of affine group schemes always respects kernels.
\end{proof}

\begin{corollary}\label{solv}
A reduced affine group scheme $G$ is solvable if and only if $G_{\perf}$ is solvable.
\end{corollary}
\begin{proof}
Lemma~\ref{CorSES} shows that for any solvable affine group scheme, its perfection is again solvable. On the other hand, assume that $G_{\perf}$ is solvable and $G$ reduced. Applying iteratively Theorem~\ref{ThmSub}(2) allows us to construct a finite chain of reduced normal subgroups such that the perfection of the quotients are abelian. A reduced affine group scheme with abelian perfection is clearly abelian itself. 
\end{proof}

\subsection{Isomorphic perfections}

We gather some examples and results about affine group schemes with isomorphic perfections.
\begin{example} \label{Ex1Perf}
\begin{enumerate}
\item Let $G$ be a finite group scheme  ({\it i.e.} $k[G]$ is finite dimensional). We have $G_{\perf}\cong G_{\red}$, so in particular, $G$ is infinitesimal if and only if $G_{\perf}$ is trivial.
\item Reduced affine group schemes can also become isomorphic after perfection. For instance, if $q$ is a power of $p$, then $(SL_q)_{\perf}\cong (PGL_q)_{\perf}$. This is an example of Lemma~\ref{LemIsog} below, or follows from \ref{Idea} below and Example~\ref{ExampRoot}. Moreover, the latter results also show that, conversely, $(SL_n)_{\perf}\cong (PGL_n)_{\perf}$ implies that $n$ must be a power of $p$.
\item As follows from Remark~\ref{remaffsch}, a necessary condition for affine group schemes $G,H$ to have isomorphic perfections is that there exists an isomorphism $G(k)\cong H(k)$ as abstract groups.
\end{enumerate}
\end{example}

Recall that an {\bf isogeny} is a faithfully flat homomorphism of affine group schemes with finite kernel $N$. An isogeny is {\bf infinitesimal} if $N$ is infinitesimal. An isogeny $G\to Q$ between reduced and connected affine group schemes is {\bf purely inseparable} if the induced morphism $k(Q)\to k(G)$ between the fields of fractions is a purely inseparable field extension.

\begin{lemma}\label{LemIsog}
For an isogeny $q:G\to Q$, the following conditions are equivalent:
\begin{enumerate}
\item[(a)] $q$ is infinitesimal;
\item[(b)] The perfection of $q$ is an isomorphism.
\end{enumerate}
Moreover, if $G,Q$ are reduced and connected, the above properties {  are equivalent to}
\begin{enumerate}
\item[(c)] $q$ is purely inseparable.
\end{enumerate}
\end{lemma}
\begin{proof}
The equivalence between (a) and (b) is an immediate application of Example~\ref{Ex1Perf}(1) and Lemma~\ref{CorSES}.

Condition (b) implies that for every $a\in k[G]$, there is $i\in\mN$ such that $a^{p^i}$ is in the image of $k[Q]\to k[G]$, from which (c) follows immediately. { Conversely, that (c) implies (b) follows similarly, by exploiting the equality 
$$k[Q]\;=\;k(Q)\cap k[G],\qquad\mbox{inside
$k(G)$.}$$}

{ To prove the displayed equality (of which the inclusion is obvious), we consider one $f\in k(G)$ which belongs to $k(Q)\subset k[G]$. By viewing $G$ as the inverse limit of quotient group schemes of finite type $\varprojlim G_\alpha$ (which induces $Q\cong \varprojlim Q_\alpha$), we can easily reduce to the case where $G$ and $Q$ are of finite type, by taking $\alpha$ for which $f\in k[G_\alpha]\subset k[G]$ as well as $f\in k(Q_\alpha)\subset k(G_\alpha)$.}

{  Assume thus that $G$ is of finite type and consider the span $S$ of $\{g(f)\,|\, g\in G(k)\}$. This is the $G(k)$-subrepresentation of the rational left regular representation $k[G]$, in particular $S$ is finite dimensional. Clearly the action of $G(k)$ on $f$
factors through the canonical action of $Q(k)$ on $k(Q)$ and by our finite type assumption we have $G(k)\tto Q(k)$. Hence the elements $h\in k[Q]$ for which $hS\subset k[Q]$ form a non-zero (by finite dimensionality of $S$) $Q(k)$-invariant ideal $I<k[Q]$ and hence $I=k[Q]$. So $f\in S\subset k[Q]$ as desired. 

Note that an alternative argument considers the purely inseparable isogeny $G\to G/(\ker q)^0$, which is also \'etale and therefore an isomorphism.}
\end{proof}

\begin{prop}\label{PropIso}
The following conditions are equivalent on two reduced affine group schemes $G,H$ of finite type.
\begin{enumerate}
\item[(a)] $G_{\perf}\cong H_{\perf}$;
\item[(b)] There exists $j\in\mN$ and an infinitesimal isogeny $G\to H^{(j)}$;
\item[(c)] There are $i,j\in\mN$ and homomorphims $\phi: G\to H^{(j)}$, $\psi: H\to G^{(i)}$ for which the following triangles are commutative 
$$\xymatrix{
G\ar[r]^-{\phi}\ar[rd]_-{\Fr^{i+j}}&H^{(j)}\ar[d]^-{\psi^{(j)}}&&H\ar[r]^-{\psi}\ar[rd]_-{\Fr^{i+j}}&G^{(i)}\ar[d]^-{\phi^{(i)}}\\
&G^{(i+j)}&& &H^{(i+j)}.
}$$

\end{enumerate}
\end{prop}
\begin{proof}
That (b) implies (a) is a special case of Lemma~\ref{LemIsog}. 

Now we prove that (c) implies (b).
The kernel of the Frobenius homomorphism is infinitesimal, hence the left diagram shows that the kernel of $\phi$ is infinitesimal. Since we assume that $G$ and $H$ are reduced, the Frobenius homomorphism is faithfully flat. The right diagram thus proves that $\phi$ is faithfully flat.

Applying Lemma~\ref{LemGHperf} to the isomorphism in (a) yields morphisms $\phi: G^{(-j)}\to H$ and $\psi: H^{(-i)}\to G$. Expressing that these induce mutually inverse homomorphisms on the perfected groups then states that there exists $l\in\mN$ such that the composition
$$H^{(-i-j-l)} \xrightarrow{\Fr^{l}} H^{(-i-j)}\xrightarrow{\psi^{(-j)} }G^{(-j)}\xrightarrow{\phi} H$$
is $\Fr^{i+j+l}$. Since $\Fr^l$ is faithfully flat, we arrive precisely at the conditions in (c), so (a) implies (c).
\end{proof}

\subsection{Tannakian point of view}

For an affine group scheme $G$ we denote by $\Rep G$ its category of (rational) representations which are finite dimensional over $k$. Its category of all representations will be denoted by $\Rep^\infty G\cong \Ind\Rep G$.

\subsubsection{} {Let us} interpret $G_{\perf}$ as
$\varprojlim G^{(-i)}$ as in \ref{relative2}. { We refer to \cite[\S 6]{CEO} for an overview of the notion of the direct limit of (tensor) categories. In particular, we have}
\begin{equation}\label{TannLim2}\Rep (G_{\perf})\;\cong\;\varinjlim \Rep G^{(-i)},\end{equation}
where the $k$-linear (exact) tensor functors in the chain are given by the pullback along $G^{(-i-1)}\to G^{(-i)}$.
These functors fit into commutative diagrams:
\begin{equation}\label{2twists}\xymatrix{
\Rep G\ar[r]\ar[dr]_{-^{(1)}}&\Rep G^{(-1)}\ar[d]^{\sim} && (V\to V\otimes k[G])\ar@{|->}[r]\ar@{|->}[rd]&(V\to V\otimes k[G]^{(-1)})\ar@{|->}[d]\\
&\Rep G&&&(V^{(1)}\to V^{(1)}\otimes k[G]) .
}\end{equation}
The non-horizontal arrows are only $k$-linear up to twist. By definition, $V\to V\otimes k[G]^{(-1)}$ comes from $k[G]\to k[G]^{(-1)}$ in \eqref{eqRelFr}. The downwards arrow is given by applying $-^{(1)}$ to both vector space and co-action. Note that we can equivalently
realise the $G$-representation $V^{(1)}$ from the bottom right in \eqref{2twists} as the subquotient of $\otimes^p V$ given by the
image of $\Gamma^pV\to S^p V$. This gives a more palatable
  definition of the Frobenius twist from the Tannakian point of view.

Diagram \eqref{2twists} allows us to realise $\Rep G_{\perf}$ alternatively as 
\begin{equation}\label{TannLim1}\Rep (G_{\perf})\;\cong\;\varinjlim \Rep G,
\end{equation} in the sprit of interpretation~\ref{relative1}. Compared to \eqref{TannLim2} we no longer have to work with twists of $G$, but we do have the drawback that the the defining functors are not $k$-linear.

\subsubsection{Notation}\label{notation}
For $M\in\Rep G$, we denote by $M[i]$ the object of $\varinjlim \Rep G$ where $M$ is placed in the $i$-th copy of $\Rep G$ in the chain, and use the same notation for the corresponding object in $\Rep(G_{\perf})$, via \eqref{TannLim1}. Note that every object in $\Rep (G_{\perf})$ is of this form and furthermore $M[i]\cong M^{(1)}[i+1]$.

\begin{lemma}\cite[Remark~6.5]{CEO}\label{LemTann}
Let $G$ be an affine group scheme over $k$.
\begin{enumerate}
\item $G$ is reduced if and only if $-^{(1)}: \Rep G\to \Rep G$ is full.
\item $G$ is perfect if and only if $-^{(1)}: \Rep G\to \Rep G$ is an equivalence.
\end{enumerate}
\end{lemma}
\begin{proof}
If $G$ is reduced, the $p$-th power map is injective on $k[G]$ and the fullness in (1) follows. On the other hand, if $-^{(1)}$ is not full then (by applying adjunction) there is a $G$-representation $V$ with a vector $v\in V$ which is not $G$-invariant, but for which $1\otimes v\in V^{(1)}$ is $G$-invariant. Looking at the $k[G]$ coaction then provides a non-zero $f\in k[G]$ with $f^p=0$.

Via diagram \eqref{2twists}, the functor is an equivalence if and only if $G^{(-1)}\to G$ is an isomorphism, which is equivalent to $G$ being perfect.
\end{proof}

\begin{remark}\label{ExtLim}
As for any direct limit of abelian categories, for objects $M[i],N[i]\in \Rep(G_{\perf})$, using notation from \ref{notation}, we have
$$\varinjlim_{j}\Ext^l_{G}(M^{(j)}, N^{(j)})\cong \Ext^l_{G_{\perf}}(M[i],N[i]).$$
\end{remark}

\subsection{Additive, multiplicative and unipotent groups}
For convenience we let $k$ be algebraically closed in this section.

\subsubsection{}To lighten expressions, we introduce the following notation
$$\bG_a:=(\mG_a)_{\perf}\quad\mbox{and}\quad \bG_m:=(\mG_m)_{\perf}$$
for the perfection of the additive and multiplicative group of $k$.

\begin{prop}\label{PropDim1}
Assume that $k$ is algebraically closed.
\begin{enumerate}
\item Let $\bG$ be a connected affine group scheme of perfect finite type and of dimension $1$, then either $\bG\cong \bG_a$ or $\bG\cong \bG_m$.
\item Let $G$ be a reduced affine group scheme of finite type with $G_{\perf}\cong \bG_a$ (resp. $G_{\perf}\cong \bG_m$), then $G\cong \mG_a$ (resp. $G\cong\mG_m$). 
\end{enumerate}
\end{prop}
\begin{proof}
By \cite[Theorem~3.4.9]{Springer}, for any connected reduced affine group scheme $G$ of finite type and of dimension $1$, we must have $G\cong \mG_a$ or $G\cong \mG_m$.
This implies part (2), by Lemma~\ref{LemBasic}(1). Part (1) follows similarly, using characterisation \ref{PerfAlgGr}(d). \end{proof}

By a torus we mean an affine group scheme isomorphic to $\mG_m^{\times n}$ for some $n\in\mN$. Similarly, {\bf perfect tori} are the affine group schemes isomorphic to $\bG_m^{\times n}$ for some $n\in\mN$.

\begin{lemma}\label{LemTor}
For a connected reduced affine group scheme $G$ of finite type, the following are equivalent:
\begin{enumerate}
\item[(a)] $G$ is a torus;
\item[(b)] $G_{\perf}$ is a perfect torus.
\end{enumerate}
\end{lemma}
\begin{proof}
Clearly (a) implies (b). On the other hand, (b) implies that $G$ is connected and, since 
$$\Rep G\to \Rep G_{\perf}$$
is exact and fully faithful, see Lemma~\ref{LemTann}(1), it follows that $\Rep G$ is semisimple and pointed (every simple representation has dimension one). That $G$ is a torus then follows from \cite[3.2.3 and 3.2.7(ii)]{Springer}.
\end{proof}

Recall that an affine group scheme $ G$ is {\bf unipotent} if and only if every representation has an invariant vector (equivalently every simple object in $\Rep G$ is trivial).


\begin{lemma}\label{CorUni}
 For an affine group scheme $G$ over $k$, $G_{\perf}$ is unipotent if and only if $G_{\red}$ is unipotent.
\end{lemma}
\begin{proof}
By equivalence~\eqref{TannLim2}, if $G$ (or $G_{\red}$) is unipotent, then so is $G_{\perf}$. On the other hand, if $G_{\perf}$ is unipotent, then the fully faithful exact functor from $\Rep G_{\red}$ to $\Rep G_{\perf}$ shows that also $G_{\red}$ is unipotent.
\end{proof}

\subsubsection{}\label{ExZ1p}
We have a ring isomorphism
\begin{equation}\label{EndGm}
\mZ[1/p]\;\xrightarrow{\sim}\;\End(\bG_m),\qquad a\mapsto \{\blam\mapsto \blam^a\},\end{equation}
where $\blam$ stands for an element of $\bG_m(A)=\varprojlim A^\times$ for a commutative $k$-algebra $A$. The restriction to $p^{\mZ}\hookrightarrow \Aut(\bG_m)$ is the homomorphism from~\ref{pZ}.

For $\mG_a$, the latter homomorphism extends to an isomorphism
\begin{equation}\label{AutGa}
k^\times \rtimes p^{\mZ} \;\xrightarrow{\sim}\; \Aut(\bG_a),\qquad (\kappa, n)\mapsto \{ \theta^n_\kappa: \blam\mapsto \kappa\blam^n\}.
\end{equation}


\subsubsection{}\label{PerfChar}
For a perfect affine group scheme $\bG$, Lemma~\ref{LemGHperf} shows that characters of $\bG$ correspond to homomorphisms $\bG\to\bG_m$. As the latter formulation carries more structure, we define
$$\bX(\bG)\,:=\, \Hom(\bG,\bG_m)\,\cong\, \Hom(\bG,\mG_m)=X(\bG).$$
This is a $\mZ[1/p]$-module via~\eqref{EndGm}. Consequently, we will define cocharacters of $\bG$ to be the $\mZ[1/p]$-module
$$\bY(\bG):=\Hom(\bG_m,\bG).$$
We have the obvious bilinear pairing
\begin{equation}\label{pairing}\bX(\bG)\times\bY(\bG)\to\mZ[1/p].\end{equation}
It is non-degenerate if $\bG$ is a perfect torus.

\subsection{Induction}

\subsubsection{} For a homomorphism $f:H\to G$ of affine group schemes, we will sometimes abbreviate $f_\ast:=\Ind^{G}_H$ and $f^\ast:=\Res^G_H$. This gives an adjoint pair $f^\ast\dashv f_\ast$ of functors between $\Rep^\infty G$ and $\Rep^\infty H$.

\subsubsection{}\label{adjunctions} For a commutative square of group homomorphisms 
$$\xymatrix{
H\ar[rr]^f\ar[d]^a&&G\ar[d]^b\\
A\ar[rr]^g&&B,}$$
the adjunction morphisms yield a natural transformation
$$\xi:b^\ast g_\ast\;\Rightarrow\; f_\ast a^\ast.$$
In particular, for $M\in \Rep^\infty A$, the morphism $\xi_M$ is zero if and only if the composite
$$g^\ast g_\ast M\;\to\; M\;\to\; a_\ast a^\ast M$$
is zero.

\subsubsection{} Now we consider two affine group schemes of the form $A=\varprojlim A_i$ and $B=\varprojlim B_i$ for inverse systems of affine group schemes $(A_i\,|\, i\in\mN)$ and $(B_i\,|\, i\in\mN)$. We label the homomorphisms $p_i: A\to A_i$, $q_i: B\to B_i$ and $a_{[i,j]}:A_{i}\to A_j$ for $i>j$. Assume also given homomorphisms $A_i\to B_i$, leading to $A\to B$.

\begin{prop}\label{IndProp}
\begin{enumerate}
\item For $M_i\in \Rep^\infty A_i$ and morphisms $a^\ast_{[i+1,i]}M_i\to M_{i+1}$, the evaluations at $M_i$ of the natural transformations in \ref{adjunctions} lead to an isomorphism
$$\varinjlim q_i^\ast \Ind^{B_i}_{A_i}M_i\;\xrightarrow{\sim}\; \Ind^B_A\varinjlim p_i^\ast M_i. $$
\item For $M\in \Rep^\infty A_j$ and $n\in\mN$, we have a canonical isomorphism
$$\varinjlim q_i^\ast ( R^n\Ind^{B_i}_{A_i}(a_{[i,j]}^\ast M))\;\xrightarrow{\sim}\; R^n\Ind^B_A (p_j^\ast M).$$
\end{enumerate}
\end{prop}
\begin{proof}
The right-hand side in part (1) is given by the $A$-invariants in the vector space
$$\varinjlim (M_i\otimes k[B_i]).$$
Since direct limits commute with co-equalisers, this is isomorphic to the direct limit of $A_i$-invariants in the above spaces. This proves part (1).

Note that the case $n=0$ in part (2) is the special case of part (1) where we set $M_i:=a^\ast_{[i,j]}M$. We can prove the analogue of part (1) for derived functors, which similarly specialises to part (2), as follows. For a chain $\{M_i\}$ as in part (1), we consider injective hulls in $\Rep^\infty A_i$, yielding short exact sequences
$$0\to M_i\to I_i\to Q_i\to 0.$$
The defining property of injective modules gives a chain map from the action of $a_{[i+1,i]}^\ast$ on the above sequence and the corresponding sequence for $i+1$.
In particular, this makes $\{I_i\}$ and $\{Q_i\}$ into chains of representations as in part (1), and we have a short exact sequence
$$0\to \varinjlim p_i^\ast M_i\to\varinjlim p_i^\ast  I_i\to\varinjlim p_i^\ast  Q_i\to 0$$
in $\Rep^\infty A$. The claim now follows by induction on $n$, using long exact sequences in homology if we observe that $I:=\varinjlim p_i^\ast  I_i$ is injective in $\Rep^\infty A$. Exactness of
$$\Hom_A(-,I):\;\Rep A\to\Vecc^\infty$$ 
follows from the observation $\Rep A\cong \varinjlim \Rep A_i$. The latter exactness is sufficient to conclude that $I$ is injective. Indeed, we can consider an injective hull $I\subset I'$ and an intermediate module $I\subset I''\subset I'$ for which $I''/I$ is finite dimensional. Now we must have $I''\cong I\oplus I''/I$ (apply $\Hom(-,I)$ to a finite submodule of $I''$ which still surjects onto $I''/I$) which violates $\soc I =\soc I'$ unless $I''=I$.
\end{proof}

\begin{corollary}\label{PerfectInduction}
Let $G$ be an affine group scheme with subgroup $H$, set $\bG=H_{\perf}$, $\bH=H_{\perf}$ and take $n\in\mN$. For $M\in \Rep H$ such that $R^n\Ind^{G}_{H}M^{(i)}$ is finite dimensional for each $i\in\mN$. For each $j\in\mN$ we have an isomorphism
$$\varinjlim_{i\ge j} \left(R^n  \Ind^{G}_{H}(M^{(i-j)})[i]\right)\;\xrightarrow{\sim}\; R^n  \Ind^{\bG}_{\bH}(M[j]),$$
with notation as in \ref{notation}.
\end{corollary}
\begin{proof}
We start by applying Proposition~\ref{IndProp}(2), using the interpretation $G_{\perf}=\varprojlim G^{(-i)}$, applied to $M^{(-j)}\in \Rep H^{(-j)}$. By assumption, all the representations appearing in the direct limit in \ref{IndProp}(2) are finite dimensional. After passing from \eqref{TannLim2} to \eqref{TannLim1} this allows us to use the notation from \ref{notation} to rewrite the isomorphism in the desired form.
\end{proof}

\begin{remark}\label{RemNot}
\begin{enumerate}
\item For group schemes of finite type, we can prove Corollary~\ref{PerfectInduction} alternatively using Theorem~\ref{ThmSub}(3) and (a generalisation from line bundles to general quasi-coherent sheaves with identical proof  of) Lemma~\ref{LemLine}.
\item Assume that $M$ is one-dimensional, and $ \Ind^G_{H}(M^{(j)})\not=0$ for all $j$. It follows from \ref{adjunctions} that the morphisms in the directed system for the left-hand side in Corollary~\ref{PerfectInduction} for $n=0$ are all non-zero.
\end{enumerate}
\end{remark}


\section{Perfectly reductive groups}\label{SecRed}

In this section we assume that $k$ is algebraically closed of characteristic $p>0$.

\subsection{Definition}

Recall that a reductive group over $k$ is a connected reduced (smooth) affine group scheme of finite type which has no non-trivial normal reduced unipotent subgroup.

\begin{theorem}\label{DefRG}
\begin{enumerate}
\item For an affine group scheme $\bG$, the following are equivalent:
\begin{enumerate}
\item $\bG$ is a connected affine group scheme of perfect finite type and has no non-trivial perfect normal unipotent subgroups;
\item $\bG\cong G_{\perf}$ for a reductive group $G$;
\item $\bG$ is an affine group scheme of perfect finite type and, { whenever $\bG\cong H_{\perf}$ for a reduced affine group scheme of finite type $H$, then $H$ must be reductive.}
\end{enumerate}
If these conditions are satisfied, we call $\bG$ {\bf perfectly reductive}.
\item For every connected affine group scheme of perfect finite type $\bH$, there exists a short exact sequence of perfect affine group schemes
$$1\to \bU\to \bH\to \bQ\to 1,$$
where $\bU$ is unipotent and $\bQ$ is perfectly reductive.
\end{enumerate}
\end{theorem}
\begin{proof}
First we show that 1(b) implies 1(a). Set $\bG=G_{\perf}$ for a reductive group $G$. Let $\bU\lhd \bG$ be a perfect normal unipotent subgroup. Then by Theorem~\ref{ThmSub}(2), there exists a reduced normal subgroup $U\lhd G$ with $U_{\perf}=\bU$. By Lemma~\ref{CorUni}, $U$ is unipotent, so $U$ is trivial. Consequently $\bU$ is trivial.

That 1(c) implies 1(b) follows from Theorem~\ref{PerfAlgGr}. 

For $\bH$ as in (2), we know that $\bH=H_{\perf}$ for a connected reduced affine group scheme of finite type $H$ by Theorem~\ref{PerfAlgGr}. By taking the perfection of the short exact sequence corresponding to the unipotent radical $R_uH\lhd H$, see \cite[\S 6.4.6]{Milne}, we get a short exact sequence as desired in (2) (provided we define, for now, perfectly reductive groups as the perfections of reductive groups), with $\bU:=(R_uH)_{\perf}$, by Lemma~\ref{CorSES}. 

Since $R_uH$ is reduced, $\bU$ is trivial if and only if $R_uH$ is trivial, which shows that 1(a) implies 1(c).
\end{proof}

\begin{remark}
In addition to Theorem~\ref{DefRG}(2), we can also observe that every affine group scheme of perfect finite type $\bG$ admits a short exact sequence $\bH\to \bG\to Q$ where $Q$ is a finite abstract group and $\bH$ is a connected affine group scheme of perfect finite type. This follows from perfecting the classical theory, see \cite[\S 2.g]{Milne}.
\end{remark}

\subsubsection{}\label{pinnings}
A {\bf perfect Borel} subgroup $\bB$ of a perfectly reductive group is a maximal solvable perfect subgroup. For a reductive group $G$, every perfect Borel subgroup of $G_{\perf}$ is the perfection of a Borel subgroup of $G$ by Theorem~\ref{ThmSub}(2) and Corollary~\ref{solv}. In particular, every two perfect Borel subgroups are conjugate. 

Similarly, by Theorem~\ref{ThmSub}(2) and Lemma~\ref{LemTor}, every maximal perfect torus $\bT$ in $G_{\perf}$ is the perfection of a maximal torus $T<G$, and so maximal perfect tori are unique up to conjugation.

We henceforth use freely that every such choice $\bT<\bB < G_{\perf}$, for a reductive group $G$, can be obtained as the perfection of a corresponding choice $T<B<G$ of Borel subgroup and maximal torus in $G$.

\subsection{Classification}
We freely use the equivalence between $D$-root data and real-type $D$-root data for $D=\mZ$ and $D=\mZ[1/p]$ from Theorem~\ref{ThmData}.

\begin{theorem}\label{ThmClass}
There is a canonical bijection between the set of isomorphism classes of perfectly reductive groups over $k$ and the set of isomorphism classes of $\mZ[1/p]$-root data.
\end{theorem}

  \begin{remark} 
\begin{enumerate}
\item As in the classical case, this theorem can be
    extended to cover isogenies. We do this in \ref{subsec:isogenies}
    below.
    \item {  Theorem~\ref{ThmClass} implies in particular that for two reductive groups to have isomorphic perfections, they must have Weyl groups that are isomorphic as Coxeter groups. Slightly more restrictive, they must have the same Dynkin diagram, except that we can have perfected isomorphisms between types B and C when $p=2$.
    
    }
\end{enumerate}

  \end{remark}

\subsubsection{Idea of the proof}\label{Idea} 
We will prove that, when characterising a reductive group in terms of its root datum, two reductive groups become isomorphic after perfection if and only if their root data become isomorphic after extension of scalars to $\mZ[1/p]$. More explicitly:

Denote by $D\mbox{-}\mathcal{R}\mathcal{D}$ the set of isomorphism classes of $D$-root data. Furthermore, we let $\mathcal{R}e\mathcal{G}r$, resp. $\mathcal{P}e\mathcal{R}e\mathcal{G}r$, denote the set of isomorphism classes of reductive groups, resp. perfectly reductive groups, over $k$.
We can exploit the classical bijection between $\RD$ and $\RG$, see for instance \cite{De}, and include it in the following (commutative) diagram:
\begin{equation}\label{reddiag}\xymatrix{
\RG\ar@{->>}[rrr]^-{-_{\perf}}\ar[d]^{1:1}&&& \pRG\ar@{-->}@/^/[d]\\
\RD\ar@{->>}[rrr]^-{\mZ[1/p]\otimes-}\ar[u]&&& \pRD.\ar@{-->}@/^/[u].
}\end{equation}

The upper surjection is given by the definition in \ref{DefRG}(1)(b) of perfectly reductive groups. The lower surjection comes from Lemma~\ref{LemQ}.
To prove Theorem~\ref{ThmClass}, it suffices to show the dashed arrows in \eqref{reddiag} exist. This is established in the following two propositions.

\begin{prop}
 If the root data of two reductive groups $G_1,G_2$ extend to isomorphic root data over $\mZ[1/p]$, then $(G_1)_{\perf}$ and $(G_2)_{\perf}$ are isomorphic. In particular, the upwards dashed arrow in \eqref{reddiag} exists.
\end{prop}
\begin{proof} Let $(X_i,R_i,Y_i,R^\vee_i)$, with $i\in\{1,2\}$, denote the root datum of $G_i$.
Consider an isomorphism of root data given by $\psi:\mZ[1/p]\otimes X_1 \xrightarrow{\sim}\mZ[1/p]\otimes X_2$. Replacing $\psi$ by $p^l\psi$ if necessary (as is allowed by Lemma~\ref{LemBeta1}), we can assume that $\psi$ restricts to an embedding $X_1\hookrightarrow X_2$. Furthermore, we get a bijection $R_1\to R_2$, by associating to $\alpha\in R_1$ the unique element $\alpha'\in R_2$ for which $\psi(\alpha)=p^j\alpha'$ for some $j\in\mZ$. Again, after replacing $\psi$ by $p^l\psi$ if necessary, we can assume $j\in\mN$. Now it follows quickly that this $X_1\hookrightarrow X_2$ satisfies the requirements to apply \cite[Theorem~1.5]{St}, which yields an isogeny $G_1\to G_2$.

We can apply the same procedure to $\psi^{-1}$ to obtain an isogeny $G_2\to G_1$. As we might need to replace $\psi^{-1}$ again by a composition with multiplication by a power of $p$, our two isogenies will not necessarily be induced by mutually inverse maps $X_1\leftrightarrow X_2$, but by maps which compose to $p^l$ times the identity for some $l\in\mN$. Uniqueness of isogenies in \cite[Theorem~1.5]{St} then states that composition of the isogenies between $G_1$ and $G_2$ yield morphisms $\phi^l\circ \Fr^l$, for isomorphisms $\phi: G_i^{(1)}\to G_i$ as in \ref{pZ} (up to possible composition with inner automorphisms). That $(G_1)_{\perf}$ and $(G_2)_{\perf}$ are isomorphic now follows from Proposition~\ref{PropIso}.
\end{proof}

Establishing the existence of the downwards dashed arrow will
  take more work. { Note that an alternative proof of this fact will be given in Section~\ref{SecBG}. As that proof moves via topology, it seems preferable to have this direct algebraic proof too.}
The following lemma can be proved by looking at the Hopf algebra morphisms.
\begin{lemma}\label{LemFrEasy}
Consider a reduced affine group scheme $H$ with a homomorphism $\phi: H^{(-i)}\to\mG_m$, such that the diagram
$$\xymatrix{
\mG_m\times\mG_a\ar[rr]^{(\lambda,\mu)\mapsto \lambda\mu}&&\mG_a\\
H^{(-i)}\times \mG_a\ar[rr]^{\Fr^i\times\id}\ar[u]^{\phi\times\id}&&H\times \mG_a\ar@{-->}[u]
}$$
can be completed with dashed arrow to a commutative square. Then $\phi$ factors through $\Fr^i:H^{(-i)}\to H$.
\end{lemma}

\begin{definition}
Let $\bG$ be a perfectly reductive group with maximal perfect torus $\bT$.
An {\bf rt-pair} is a pair $(x,\alpha)$ of a subgroup inclusion $x:\bG_a\to \bG$ and $\alpha\in \bX$ ({\it i.e.} $\alpha:\bT\to\bG_m$), for which the following square is commutative
\begin{equation}\label{eqrt}\xymatrix{\bG_m\times\bG_a\ar[rrr]^{(\blam,\bmu)\mapsto \blam\bmu}&&&\bG_a\ar@{^{(}->}[d]^x\\
\bT\times\bG_a\ar[rrr]^{(t,\bmu)\mapsto tx(\bmu)t^{-1}}\ar[u]^{\alpha\times \id}&&&\bG.}\end{equation}
\end{definition}

If we apply this definition to ordinary reductive groups (we replace every perfect group in \eqref{eqrt} by its finite type analogue), we get precisely the pairs of inclusions of root subgroups and their corresponding root. Since root subgroups are unique, the inclusion of the root subgroup is unique up to scalar in $k^\times\cong \Aut(\mG_a)$.

\begin{lemma}\label{Lemrt}

Let $\bG$ be a perfectly reductive group with maximal perfect torus $\bT$.
\begin{enumerate}
\item The group $k^\times\rtimes p^{\mZ}$ acts on the set of rt-pairs as
$$(x,\alpha)\;\stackrel{(\kappa,n)}{\longmapsto}\; (x\circ(
\theta^n_{\kappa})^{-1},n\alpha),\qquad\mbox{for all $(\kappa,n)\in
  k^\times\rtimes p^{\mZ} $}.$$
(See \ref{ExZ1p} for the definition of $\theta^n_{\kappa}$
  and \ref{PerfChar} for the action of $n\in\mZ[1/p]$ on $\bX$.)
\item If for one $\alpha:\bT\to\bG_m$ we have two rt-pairs $(x,\alpha)$ and $(z,\alpha)$, then $z=x\circ \theta^1_\kappa$ for some $\kappa\in k^\times$.
\item Consider a reductive group $G$ with $G_{\perf}\cong \bG$, with maximal torus $T$ which perfects to $\bT$. Consider an rt-pair $(x,\alpha)$ for which $x$ is the perfection of some $x_0:\mG_a\hookrightarrow G$, then $\alpha$ is the perfection of a root homomorphism $T\to\mG_m$ and $x_0$ is an inclusion of the corresponding root subgroup.

\item For every rt-pair $(x,\alpha)$, there exists $n\in p^{\mZ}$ such that $x\circ \theta_1^n$ is the perfection of the inclusion of a root subgroup $\mG_a\hookrightarrow G$ and $n^{-1}\alpha$  is the perfection of the corresponding root homomorphism $T\to\mG_m$.
\end{enumerate}
\end{lemma}
\begin{proof}
Part (1) follows from a direct calculation.

For part (3), the homomorphism $\alpha:\bT\to \bG_m$ is induced from $T^{(-i)}\to \mG_m$ for some $i\in\mN$, as in Lemma~\ref{LemGHperf}.  We find a diagram
$$\xymatrix{
\mG_m\times\mG_a\ar[rr]&& \mG_a\ar@{^{(}->}[d]\\
T^{(-i)}\times\mG_a\ar[r]^-{\Fr^i\times \id}\ar[u]& T\times \mG_a\ar[r]& G
}$$
which `perfects' to 
diagram~\eqref{eqrt}. More precisely, after perfecting the above diagram and removing the automorphism which is the perfection of $\Fr^i\times \id$, we recover \eqref{eqrt}. In particular, we find that the above diagram is commutative.
It now follows from Lemma~\ref{LemFrEasy} that $\alpha$ comes from $\alpha_0:T\to\mG_m$ and it follows immediately that $\alpha_0$ is a root.

Now we prove part (4). By Theorem~\ref{ThmSub}(2) and Proposition~\ref{PropDim1}(2), there exists a group monomorphism $\mG_a\hookrightarrow G$ which perfects to $x\circ \phi^{n}_\kappa$ for some $n\in p^{\mZ}$,$\kappa\in k^\times$. By identifying $k^\times$ with $\Aut(\mG_a)$, we might as well take $\kappa=1$. The claim about $\alpha$ now follows from parts (1)
and (3).

Finally, we prove part (2). Since roots of reductive groups cannot be multiples of one another, part (4) implies that there is $n\in p^{\mZ}$ for which both $x\circ \theta_1^n$ and $z\circ \theta_1^n$ are the perfections of inclusions of the same root subgroup. Those inclusions must be the same, up to a scalar in $k^\times\cong \Aut(\mG_a)$, from which the claim follows.
\end{proof}

\begin{prop}\label{PropRed2}
Consider a reductive group $G$ corresponding to a root datum $RD$. One can extract $\mZ[1/p]\otimes RD$ from the group $G_{\perf}$. In particular, the downwards dashed arrow in \eqref{reddiag} exists.
\end{prop}
\begin{proof}
We construct a $\mZ[1/p]$-root datum $(\bX,\bR,\bY,\bR^\vee)$ from $\bG:=G_{\perf}$. It will follow from the construction that $(\bX,\bR,\bY,\bR^\vee)$ is isomorphic to $\mZ[1/p]\otimes RD$.

First we let $\bT$ be a maximal perfect torus in $\bG$. Recall $\bT$ is the perfection of a maximal torus $T<G$ (and hence unique up to conjugation). We choose such a $T$.
We set $\bX=\bX(\bT)$ and $\bY=\bY(\bT)$, with (non-degenerate) pairing~\eqref{pairing}. We denote the root datum of $G$ corresponding to $T$ by $(X,R,Y,R^\vee)\cong RD$.

Following Lemma~\ref{LemGHperf} we find
\begin{equation}\label{XbX}\mZ[1/p]\otimes X=\mZ[1/p]\otimes\Hom(T,\mG_m))\xrightarrow{\sim}\varinjlim\Hom(T^{(-i)},\mG_m))\xrightarrow{\sim}\Hom(\bT,\bG_m)=\bX.\end{equation}

We define $\bR\subset \bX$ as the set of $\alpha$ for which there exists an rt-pair $(x,\alpha)$. It follows from Lemma~\ref{Lemrt}(1) and (4) that $\bR\subset p^{\mZ}R$, under $R\subset X\subset \bX$ from~\eqref{XbX}. That $R\subset\bR$, so $p^{\mZ}R\subset\bR$ by Lemma~\ref{Lemrt}(1), follows from perfecting the root homomorphisms into $G$. In conclusion, $\bR=p^{\mZ}R$.

Having defined $(\bX,\bR,\bY)$, we now define the injection $-^{\vee}:\bR\to \bY$ completing the root datum. Choose $\alpha\in R\subset \bR$. Denote by $\bH$ the minimal perfect subgroup of $\bG$ which contains the images of the two morphisms $\bG_a\to\bG$ corresponding to $\alpha,-\alpha$. That the images only depend on $\alpha,-\alpha$ follows from Lemma~\ref{Lemrt}(2). That there exists such a minimal $\bH<\bG$ follows from the noetherian property of $G$ and Theorem~\ref{ThmSub}(2). The latter also shows that $\bH$ must be isomorphic to $(SL_2)_{\perf}$ or $(PGL_2)_{\perf}$ and that $\bT\cap\bH$ is a (maximal) torus in $\bH$. We use this to define $\bG_m\to \bT$, determined up to $\Aut(\bG_m)=\mZ[1/p]^{\times}$, either as the inclusion of this maximal torus of $(SL_2)_{\perf}$ or by similarly restricting the homomorphism $(SL_2)_{\perf}\tto (PGL_2)_{\perf}\hookrightarrow \bG$. Finally, we can then define $\alpha^\vee:\bG_m\to\bT$ as the unique such morphism for which composition with $\alpha:\bT\to\bG_m$ yields $2\in \End(\bG_m)$. By construction, $\alpha^\vee$ is defined independently of $G$, but clearly corresponds to the direct definition via $\bG=G_{\perf}$. For $p^i\alpha$, we set $(p^i\alpha)^\vee=p^{-i}\alpha^\vee$, which extends the definition to $\bR=p^{\mZ}R$.
\end{proof}

\begin{remark}
An alternative to the proof of Proposition~\ref{PropRed2} is given by the proof of $(a)\Rightarrow (b)\Rightarrow(c)$ in Theorem~\ref{ThmBG}. However, the latter uses deep results about 2-compact groups and \'etale homotopy types of reductive groups, hence it is preferable to have this direct proof.
\end{remark}

\subsection{Isogenies} \label{subsec:isogenies}
We establish a connection between isogenies of perfectly reductive groups and our notion of isogenies of $\mZ[1/p]$-root data from Section~\ref{SecIso1}.

\begin{theorem}
Consider two perfectly reductive groups $\bG_1$ and $\bG$ with perfect maximal tori $\bT_1,\bT$ and the corresponding $\mZ[1/p]$-root data $(\bX_1,\bR_1,\bY_1,\bR^\vee_1)$ and $(\bX,\bR,\bY,\bR^\vee)$.
There is a bijection between the sets of
\begin{enumerate}
\item[(a)] Isogenies $(\bX,\bR,\bY,\bR^\vee)\to (\bX_1,\bR_1,\bY_1,\bR^\vee_1)$;
\item[(b)] Equivalence classes of isogenies $\bG_1\to \bG$ which send $\bT_1$ to $\bT$, where two isogenies are equivalent if one is obtained from the other by composition with an inner automorphism effected by an element of $\bT_1(k)=T_1(k)$.
\end{enumerate}
\end{theorem}
\begin{proof}
Let $G,G_1$ denote reductive groups which perfect to $\bG,\bG_1$, with maximal tori $T,T_1$ which perfect to $\bT,\bT_1$, and denote their root data by  $(X,R,Y,R^\vee)$ and $ (X_1,R_1,Y_1,R^\vee_1)$.

We will prove both sets (a) and (b) are in bijection with the set of
\begin{enumerate}
\item[(c)] Equivalence classes of pairs $(\theta,i)$ of an isogeny $\theta:(X,R,Y,R^\vee)\to (X_1,R_1,Y_1,R^\vee_1)$ (in the sense of \cite[\S 1]{St}) and $i\in\mZ$, where the equivalence relation is generated by $(\theta,i)\sim (p\theta,i-1)$.
\end{enumerate}

That sending a pair $(\theta,i)$ to the extension of scalars along $\mZ\to\mZ[1/p]$ of $p^i\theta$ yields a bijection between (c) and (a) can be easily derived from Remark~\ref{RemIsog}.

An isogeny $\bG_1\to \bG$ yields a faithfully flat morphism $G^{(-i)}_1\to G$ for high enough $i$, via Lemma~\ref{LemGHperf}. Its kernel is an affine group scheme of finite type which perfects to a finite group scheme. It must therefore be a finite group scheme and $G^{(-i)}_1\to G$ is an isogeny. This principle allows us to establish a bijection between the set of isogenies $\bG_1\to \bG$ and the set of equivalence classes of isogenies $G_1^{(-i)}\to G$, with equivalence generated by the condition that $\phi:G_1^{(-i)}\to G$ be equivalent to $\phi\circ \Fr: G_1^{(-i-1)}\to G$.

The above connection between isogenies $\bG_1\to \bG$ and equivalence classes of isogenies $G_1\to G$ allows us to use the classical Isogeny Theorem \cite[1.5]{St} to establish the bijection between (b) and (c).
\end{proof}

\section{Perfected representation theory}\label{SecRep}

Let $k$ be an algebraically closed field of characteristic $p$.

\subsection{Simple and induced modules and block structure}

Let $\bG$ be a perfectly reductive group, $\bB$ a perfect Borel subgroup and $\bT <\bB$ a maximal perfect torus. 
We consider the set $\bR^+\subset\bR$ of positive roots, which are the ones for which the corresponding $\bG_a\to \bG$ {\bf does not} land in $\bB$ ({\it i.e.} we let $\bB$ be the negative Borel).

We set $\bX=\bX(\bT)$ and $\bX_+\subset\bX$ the subset of $\blam\in\bX$ which satisfy $\langle \blam,\alpha^\vee\rangle\ge 0$ for all $\alpha\in\bR^+$. We have a canonical bijection $\{\blam\mapsto k_{\blam}\}$ between $\bX$ and the set of isomorphism classes of simple $\bB$-representations (which are all one-dimensional). 

\begin{theorem}\label{ThmSimples}
\begin{enumerate}
\item The representation
$$\bnab(\blam):=\Ind^{\bG}_{\bB}k_\blam\;\in\;\Rep^{\infty}\bG$$
is zero if $\blam\not\in \bX_+$. If $\blam\in \bX_+$, it has simple socle, which we denote by $\bL(\blam)$.
\item The above association $\blam\mapsto \bL(\blam)$ is a bijection between $\bX_+$ and the set of isomorphism classes of simple representations in $\Rep \bG$.
\end{enumerate}
\end{theorem}

\begin{proof}
We choose a reductive group $G$ with $G_{\perf}\cong \bG$ and maximal torus and Borel subgroup $T< B< G$ which perfect to $\bT$ and $\bB$.  
We will use the notation in \ref{notation}.
By equivalence~\eqref{TannLim1}, every simple object in $\Rep \bG$ is of the form $L(\mu)[i]$ for some $\mu\in X_+$ and $i\in\mN$. Since $L(\mu)^{(1)}\cong L(p\mu)$, see \cite[II.3]{Jantzen}, we can define unambiguously the simple object $\tilde{L}(p^{-i}\mu):=L(\mu)[i]$ for $i\in\mN$ and $\mu\in X_+$. This clearly gives a bijection between $\bX_+$ and the set of isomorphism classes of simple objects in $\Rep \bG$.

By Corollary~\ref{PerfectInduction}, we have
\begin{equation}\label{LimNabla}\bnab(\blam)\;\cong\;\varinjlim \nabla(p^i\blam)[i],\end{equation}
where the chain of which we take the limit starts at $i$ where $p^i\blam\in X\subset\bX$. It follows now from \cite[II.2]{Jantzen} that $\bnab(\blam)=0$ whenever $\blam\not\in \bX_+$.

Now consider $\blam\in\bX_+$. The morphisms in the direct system for \eqref{LimNabla} are not zero by Remark~\ref{RemNot}(2). Hence, the defining morphisms $\nabla(p^i\blam)^{(1)}\to \nabla(p^{i+1}\blam)$ are unique (up to scalar) and injective since $\nabla(p^i\blam)^{(1)}$ has socle $L(p^{i+1}\blam)$ by Remark~\ref{ExtInj} below and  $\nabla(p^{i+1}\blam)$ is the injective hull of this socle in a Serre subcategory containing $\nabla(p^i\blam)^{(1)}$. In particular, for $\blam,\bmu\in\bX_+$, we have
$$\Hom_{\bG}(\tilde{L}(\bmu), \bnab(\blam))\cong\varinjlim \Hom_{G}(L(p^i\bmu),\nabla(p^i\blam))\cong k\delta_{\blam,\bmu}.$$
We can therefore identify $\tilde{L}(\bmu)$ with the socle of $\bnab(\bmu)$.
\end{proof}

\begin{remark}\label{RemDL}
\begin{enumerate}
\item We can alternatively construct the modules $\bnab(\blam)$ as the global sections of line bundles on $\bG/\bB\cong (G/B)_{\perf}$.
\item It follows for instance from the proof of Theorem~\ref{ThmSimples} that
$$[\bnab(\blam):\bL(\bmu)]=[\bnab(p^j\blam):\bL(p^j\bmu)],\qquad\mbox{for all }\; \blam,\bmu\in \bX_+\mbox{ and }j\in\mZ.$$
In fact, the proof even shows that
\begin{equation}\label{eqLim}
[\bnab(\blam):\bL(\bmu)]\;=\;\lim_{m \to \infty}
[\nabla(p^m\lambda):L(p^m\mu)].
\end{equation}
\end{enumerate}
\end{remark}

 {The sequence on the right-hand side in equation~\eqref{eqLim} is monotone increasing, and one can ask when it is bounded. For $G$ of rank $1$ it is clearly bounded, but the following rank 2 example was communicated to us by Stephen Donkin.}

 {
\begin{example}[Donkin] \label{ex:Donkin}
Set $p=2$ and let $E$ be a 4-dimensional space with symplectic form, so that $E$ is the natural representation of $Sp(E)\simeq Sp_{4}$. Let $\lambda$ be such that $E\simeq L(\lambda)$. Then
$$[\nabla(2^m\lambda):L(0)]\;=\; 2^{m-1}+1,\quad\mbox{so}\quad [\bnab(\blam):\bL({\mathbf 0})]\;=\;\infty .$$
Indeed, more generally $\nabla(m\lambda)$ is the symmetric power $S^mE$, and 
by \cite[Lemma~4.6]{EK}, we have
\begin{eqnarray*}
[S^{2m}E:L(0)]&=&[(S^{m-2}E)^{(1)}:L(0)]+[(S^mE)^{(1)}:L(0)]+[\wedge^2E\otimes (S^{m-1}E)^{(1)}:L(0)]\\
&=&[S^{m-2}E:L(0)]+[S^mE:L(0)]+2[S^{m-1}E:L(0)],
\end{eqnarray*}
where we used that $[\wedge^2E]=2[L(0)]+[L(\omega)]$,
with $\omega$ the second fundamental weight, so that $[L(\omega)\otimes M^{(1)}:L(0)]=0$ for all representations $M$.
It follows by induction on $m$ that
$$[S^{2m}E:L(0)]\;=\;m+1.$$
\end{example}
}

\begin{remark}\label{ExtInj}
Let $G$ be a reductive group. The canonical morphism
$$\Ext^\ast_{G}(M,N)\to \Ext^\ast_{G}(M^{(1)},N^{(1)})$$
is injective. This is proved in \cite[II.10.14]{Jantzen} if $(p-1)\rho\in X$ (e.g. $p\not=2$).

If $\rho\not\in X$, we can extend $X\subset X'\subset \mQ\otimes X$ by taking the lattice maximal $X'$ for which $\langle -,\alpha^\vee\rangle$ still takes values in $\mZ$ for every $\alpha^\vee\in R^\vee$. In particular $\rho\in X'$. Taking the appropriate $Y'\subset Y$ yields a new root datum $(X',R,Y', R^\vee)$ corresponding to a reductive group $G'$. By \cite[Theorem~1.5]{St}, there is an isogeny $G'\to G$. The claim then reduces to the previous case via an obvious commutative square.

\end{remark}

As in any category of finite dimensional modules over a coalgebra over a field, the blocks in $\Rep\bG$ are determined by the first extensions between simple objects, { see \cite[\S ii.7.1]{Jantzen}.}

\begin{theorem}\label{ThmBlock}
For $\blam,\bmu\in\bX^+$, the simple representations $\bL(\blam)$ and $\bL(\bmu)$ are in the same block of $\Rep\bG$ if and only if $\blam-\bmu\in \mZ[1/p]\bR=\mZ\bR$.
\end{theorem}
\begin{proof}
We resume the notation and conventions from the proof of Theorem~\ref{ThmSimples}. By Remarks~\ref{ExtLim} and~\ref{ExtInj}, $\bL(\blam)$ and $\bL(\bmu)$ are in the same block if and only if there exists $j\in\mN$ for which $p^j\blam,p^j\bmu\in X$ and $L(p^j\blam)$ and $L(p^j\bmu)$ are in the same block of $\Rep G$

Assume first that $\bL(\blam)$ and $\bL(\bmu)$ are in the same block. By the above, $p^j\blam-p^j\bmu\in \mZ R$ for some $j$, from which 
$\blam-\bmu\in \mZ[1/p]R= \mZ[1/p]\bR$ follows.

Now assume that $\blam-\bmu\in \mZ[1/p]\bR$. Then there exists $j\in\mN$ such that 
\begin{itemize}
\item[(i)] $p^j\blam,p^j\bmu\in X$, 
\item[(ii)] $p^j\blam-p^j\bmu\in p\mZ R$,
\item[(iii)] $\langle \alpha^\vee,p^j\blam+\rho\rangle\not\in p\mZ$ for some $\alpha\in R$.
\end{itemize} 
Indeed, for (iii) it suffices to take $\alpha$ simple and $j$ such that $p^j\blam\in pX$.

From (ii), it follows that $p^j\blam$ and $p^j\bmu$ are in the same ($\rho$-shifted) orbit of $p\mZ R\rtimes W$, so by (iii) and \cite[II.7.2(2)]{Jantzen}, $L(p^j\blam)$ and $L(p^j\bmu)$ are in the same block of $\Rep G$ from which the conclusion follows.
\end{proof}

\begin{remark}
Theorem~\ref{ThmBlock} can be explained by the observation that the orbits of the affine Weyl group $W\ltimes p\mZ R$ on $X$ describe most of the block decomposition in $\Rep G$, and the orbits
of
$ W\ltimes p\mZ[1/p]R$ on $\bX$ coincide with those of $\mZ[1/p]R$.
\end{remark}

\subsubsection{}\label{WeylType} Let $n$ be the length of the longest element $w_0$ of the Weyl group ({\it i.e.} the dimension of $\bG/\bB$).  We also set $$\bX_{++}\;=\;\{\blam\in \bX\,|\, \langle \blam,\alpha^\vee\rangle>0\mbox{ for all $\alpha\in \bR^+$}\}\subset\bX_+.$$ 
Another class of $\bG$-representations which seems of interest are
$$\bW(\blam)\;:=\; R^n\Ind^{\bG}_{\bB} k_{w_0(\blam)}\;\cong\;\varinjlim_i \Delta(p^i\blam -2\rho)[i],\quad\blam\in\bX_{++}$$
where the isomorphism is an instance of Corollary~\ref{PerfectInduction}, using \cite[II.5.11, Remark (1)]{Jantzen}.

{ Note that we do not use the notation ${\bm \Delta}(\blam)$ for the module $\bW(\blam)$ as the former would more logically be reserved for a pro-object dual to the ind-object $\bnab(\blam)$.}

\subsection{Generic cohomology}
In this section we show how the result from \cite{CPSV} can be formulated very elegantly in terms of perfected groups. { We also refer to \cite{BNP} for a more modern treatment of generic cohomology with sharper bounds.}

\subsubsection{}Let $G$ be an affine group scheme over $k$, which is the extension of scalars of an affine group scheme over $\mF_p$, which we also denote by $G$. Set $\mF=\overline{\mF_p}$. We can consider the forgetful functor from rational $G$-representations to representations over $k$ of the abstract group $G(\mF)$ (for instance using the homomorphism $G(\mF)\to G(k)$), which induces comparison morphisms
$$\Ext^i_{G}(M,N)\;\to\; \Ext^i_{k G(\mF)}(M,N).$$

\begin{theorem}\label{ThmGen}
Let $\bG$ be a perfectly reductive group, then the morphism
$$\Ext^i_{\bG}(M,N)\;\to\; \Ext^i_{k \bG(\mF)}(M,N)$$
is an isomorphism for all $M,N\in\Rep \bG$ and $i\in\mN$.
\end{theorem}

We start the proof by pointing out a technical generality.
\begin{remark}\label{RemLim}
Consider a chain of finite abstract groups $\{H_n\,|\,n\in\mN\}$ and set $H=\varinjlim H_n$. For two finite dimensional $H$-representations $M,N$, the canonical morphism
$$\Ext^i_{kH}(M,N)\;\to\; \varprojlim \Ext^i_{kH_n}(M,N)$$
is an isomorphism. Indeed, using group cohomology, $\Ext^i_{kH}(M,N)$ is the cohomology of the inverse limit of chain complexes with cohomology $\Ext^i_{kH_n}(M,N)$. Since all vector spaces involved are finite-dimensional, the Mittag-Leffler property leads to the conclusion.
\end{remark}

\begin{proof}[Proof of Theorem~\ref{ThmGen}]
Let $G$ be a reductive group with perfection $\bG$ and recall that
$\bG(\mF)=G(\mF)$ and $\Rep \bG\cong \varinjlim \Rep G$. Without loss of generality we assume that $M, N$
  factor over the natural map $\bG \to G$.
By Remark~\ref{ExtLim}, we have
$$\Ext^i_{\bG}(M,N)\;\cong\; \varinjlim_{a}\Ext^i_{G}(M^{(a)},N^{(a)}).$$
It is proved in \cite{CPSV} that the directed system in the above limit stabilises and moreover,  {  for fixed $M,N$, for large enough $a$ and $q$, all morphisms 
$$\Ext^i_{G}(M^{(a)},N^{(a)})\to \Ext^i_{kG(\mF_q)}(M^{(a)},N^{(a)})$$ are isomorphisms. Note that $G(\mF_q)$, being a finite abstract group, is perfect in the group scheme sense. Hence, for large enough $q$, we find the composite isomorphism
$$ \varinjlim_{a}\Ext^i_{G}(M^{(a)},N^{(a)})\;\xrightarrow{\sim}\;  \varinjlim_{a}\Ext^i_{k G(\mF_q)}(M^{(a)},N^{(a)})\;\xleftarrow{\sim}\;\Ext^i_{k G(\mF_q)}(M,N).$$}

Hence also the inverse system in 
$$\varprojlim \Ext^i_{k G(\mF_{p^n})}(M,N)\;\cong\; \Ext^i_{kG(\mF)}(M,N)$$ stabilises and we find the isomorphism in the theorem.

Strictly speaking, \cite{CPSV} only deals with semisimple groups. However, for a general reductive group $G$, we have a short exact sequence $N\to G\to G/N$ with $G/N$ semisimple and $N$ a torus. Since both $N$ and $N(\mF_q)\cong C_{q-1}^{\times r}$ have semisimple representation theory over $k$, the result extends easily, for instance via a collapsing Hochschild-Serre spectral sequence.
\end{proof}

\begin{remark}
The comparison morphism in Theorem~\ref{ThmGen} is not always an isomorphism for arbitrary perfect groups, for instance
$$\Ext^1_{\bG_a}(k,k)\;\to\; \Ext^1_{k \mF^+}(k,k)$$
is the canonical inclusion { 
$$k\mZ\;\hookrightarrow\; k\hat{\mZ}\cong k\mathrm{Gal}(\mF:\mF_p)\;\hookrightarrow\; \varprojlim k \mathrm{Gal}(\mF_q:\mF_p),$$
where we used Remark~\ref{RemLim} and the linear isomorphism (with $\iota:\mF_q\hookrightarrow k$)
$$k \mathrm{Gal}(\mF_q:\mF_p)\xrightarrow{\theta\mapsto \iota\circ \theta} \Hom(\mF_q^+,k^+)\cong H^1(\mF_q^+,k).$$}
\end{remark}

\begin{question}\label{QExt}
The formulation of Theorem~\ref{ThmGen} suggests the question of whether the monomorphism (by Theorem~\ref{ThmGen})
$$\Ext^i_{\bG}(M,N)\;\to\; \Ext^i_{k \bG(k)}(M,N)$$
is also an isomorphism. This is equivalent to the question of whether the epimorphism
\begin{equation}\label{RemPerf}
\Ext^i_{kG(k)}(M,N)\;\to\; \Ext^i_{k G(\mF)}(M,N)
\end{equation}
is an isomorphism for all $M,N\in\Rep G$ and $i\in\mN$. Note that \eqref{RemPerf} does not involve any perfection.
\end{question}

\begin{example}\label{ExGm}
The question in \ref{QExt} has an affirmative answer for $G=\mG_m$. Consider the short exact sequence
$$1\to \mG_m(\mF)\to \mG_m(k) \to Q\to 1.$$
By the Lyndon-Hochschild-Serre spectral sequence, showing \eqref{RemPerf} is an isomorphism can be quickly reduced to showing the group cohomology $H^i(Q,k)$ is zero for $i>0$. Now the group structure on $Q$ extends (uniquely) to a $\mQ$-vector space (since  $\mG_m(\mF)< \mG_m(k)$ is the group of roots of unity and $k$ is algebraically closed), so we only need to show $H^i(\mQ,k)=0$ for $i>0$. The case $i=1$ is obvious. 
One can compute directly that $H_i(\mQ,-)=0$ for $i>1$ (or via $B\mQ$, see~\cite[(10) on p42]{Su}), hence $H^i(\mQ,-)=0$ for $i>2$. Finally, 
$$H^2(\mQ,k)\;\cong\;\Ext^1_{\mZ}(\mQ,k)$$
must be an abelian group admitting both the structure of a $\mQ$-vector space as well as a $k$-vector space, hence it is zero.
\end{example}



We conclude with an example showing that \eqref{RemPerf} being an isomorphism is also something which should not be expected to hold outside of reductive groups.
\begin{example}
If instead of a reductive group, we consider $G=\mG_a$, as well as $i=1$, $M=N=k$ in \eqref{RemPerf}, we obtain the morphism between spaces of group homomorphisms
$$\End(k^+)\to\Hom(\mF^+,k^+),$$
induced by restriction along $\mF\hookrightarrow k$. This is not a bijection as soon as $\mF\not=k$.
\end{example}


\section{Localisation of classifying spaces}\label{SecBG}
Fix a prime $p$.
\subsection{Main result}

Following \cite{Su}, by a {\bf simple space} we mean a connected topological space having the homotopy type of a CW complex and abelian fundamental group which acts trivially on the homotopy and homology of the universal covering space.
Let $F$ be a connected topological group of homotopy type of a CW
complex (below we will consider more specifically complex Lie groups),
then its classifying space $BF$ is a simple space ({note that} $\pi_1(BF)=\pi_0(F)$ is trivial).
Following \cite[Chapter~2]{Su}, to each simple space $X$ we can associate a (simple) space $X_{\frac{1}{p}}$, the localisation of $X$ away from $p$. Note that {\it loc. cit.} $X_{\frac{1}{p}}$ is denoted by $X_{\ell}$ where $\ell$ is the set of all primes different from $p$.

\begin{theorem}\label{ThmBG}
Let $G, H$ be split (connected) reductive groups over $\mZ$. The following are equivalent:
\begin{enumerate}
\item[(a)] The perfections of $G_k$ and $H_k$ are isomorphic for $k=\overline{\mF}_p$;
\item[(b)] The localisations $BG(\mC)_{\frac{1}{p}}$ and $BH(\mC)_{\frac{1}{p}}$ are homotopy equivalent;
\item[(c)] The root data of $G$ and $H$ become isomorphic after extension to $\mZ[1/p]$.
\end{enumerate}
\end{theorem}
\begin{proof}[Proof of $(c)\Leftrightarrow (a)\Rightarrow (b)$]
The equivalence of (a) and (c) is already established in \ref{Idea}.

Assume that the perfections of $G_k$ and $H_k$ are isomorphic. By Proposition~\ref{PropIso}, after replacing $H_k$ with an (isomorphic) Frobenius twisted version, there is an infinitesimal isogeny $G_k\to H_k$. By Lemma~\ref{LemIsog}, this isogeny is purely inseparable. It then follows from \cite[Theorem~1.6]{Fr} that $BG(\mC)_{\frac{1}{p}}$ and $BH(\mC)_{\frac{1}{p}}$ are homotopy equivalent.
\end{proof}

The rest of this chapter is devoted to the proof of $(b)\Rightarrow (c)$.

\subsection{Some useful facts}

\begin{enumerate}[(a)]
\item For a complex reductive group $F$ and a maximal compact subgroup $K< F$ (the corresponding compact connected Lie group), the homomorphism $K\to F$ is a homotopy equivalence and hence $BK\simeq BF$. We will therefore henceforth replace $BF$ by $BK$.
\item For a simple space $X$, the defining map $X\to X_{\frac{1}{p}}$, see \cite[Chapter~2]{Su}, induces an isomorphism $$H_\ast(X;\mZ)\otimes\mZ[1/p]\xrightarrow{\sim}H_\ast(X_{\frac{1}{p}};\mZ).$$
\item For a commutative ring $D$ in which $p$ is invertible, by (b) and the universal coefficient theorem, we have a natural isomorphism $H^\ast(X_{\frac{1}{p}};D)\cong H^\ast(X;D)$. In particular, a map $X_{\frac{1}{p}}\to Y_{\frac{1}{p}}$ induces a graded algebra morphism $H^\ast(Y;D)\to H^\ast(X;D)$.
\item For a flat morphism $A\to B$ of commutative algebras and a topological space $X$ for which $H_i(X;\mZ)$ are all finitely generated, $H^\ast(X;A)\otimes_AB\to H^\ast(X;B)$ is an isomorphism.
\item For a torus $T\cong (S^1)^{\times r}$, $H^\ast(BT;\mZ)$ is a polynomial ring with $r$ generators in degree $2$. Consequently, for any commutative ring $R$, $H^\ast(BT;R)\cong H^\ast(BT;\mZ)\otimes R$ is a polynomial ring and we have a bijection between $R$-linear morphisms $\theta:H^2(BT';R)\to H^2(BT;R)$ and graded $R$-algebra morphisms $\widetilde{\theta}:H^\ast(BT';R)\to H^\ast(BT;R)$.
\item For a compact connected Lie group $G$, with maximal torus $T$, and a commutative ring $R$, consider canonical isomorphisms
$$R\otimes X(T)\cong R\otimes H^1(T;\mZ)\cong R\otimes H^2(BT;\mZ)\cong H^2(BT;R).$$
The Weyl group $W$ thus acts $R$-linearly on $H^2(BT,R)$. Moreover the image of 
$$H^\ast(BG;R)\to H^\ast(BT;R)$$
takes values in the algebra of $W$-invariants, see \cite[\S 27]{Borel}.
\item 
For a prime $q$ and a connected CW complex $Y$, we denote by $Y_{\hat{q}}$ the profinite completion at $q$ of $Y$, see \cite[Chapter~3]{Su}. If  $q\not=p$, then the universality of $X\to X_{\frac{1}{p}}$ in the definition in \cite[Chapter~2]{Su} shows that the latter map induces a homotopy equivalence $X_{\hat{q}}\simeq (X_{\frac{1}{p}})_{\hat{q}}$. 
\item For $G$ a compact connected Lie group, $BG$ satisfies the
  requirement in (c), {\it i.e.} the homology groups $H_i(BG;\mZ)$ are
  finitely generated. One can observe this for instance via induction
  on $i$ using the Serre fibration $G\to EG\to BG$. Note that $EG$ is
  contractible, $BG$ is simply connected and $G$ is a finite cell complex. The Leray-Serre spectral sequence thus implies that the trivial group can be obtained, starting from $H_i(BG;\mZ)$ by a finite iteration of taking kernels of morphisms to finitely generated groups (subquotients of $H_a(BG;H_b(G))$, with $a<i$). Consequently, $H_i(BG;\mZ)$ must also be finitely generated.
\end{enumerate}

\subsection{Some results of Adams and Mahmud}
We reformulate some results of Adams and Mahmud in the form we will need.

\begin{theorem}[Adams - Mahmud]
\label{ThmAM}
Let $G$ and $G'$ be two compact connected Lie groups, with maximal tori $T,T'$ and Weyl groups $W,W'$.

\begin{enumerate}
\item For a map $f: (BG)_{\frac{1}{p}}\to (BG')_{\frac{1}{p}}$, there exists a $\mZ[1/p]$-linear morphism 
$$\theta: H^2(BT';\mZ[1/p])\to H^2(BT;\mZ[1/p])$$ yielding a commutative diagram of graded algebra homomorphisms
$$\xymatrix{H^\ast(BG;\mZ[1/p])\ar[d]&& H^\ast(BG';\mZ[1/p])\ar[ll]\ar[d]\\
H^\ast(BT;\mZ[1/p])&& H^\ast(BT';\mZ[1/p])\ar[ll]_{\widetilde{\theta}},
}$$
where the upper horizontal arrow is induced from $f$ as in (c) and $\widetilde{\theta}$ is as in (e).

Moreover, for every $w\in W$ there exists $x\in W'$ with $w\theta=\theta x$.

\item Let $D$ be an integral domain of characteristic zero and consider two $D$-linear morphisms 
$$\theta_1,\theta_2:\;H^2(BT';D)\to  H^2(BT;D)$$ for which the following composite is equal
$$\xymatrix{&& H^\ast(BG';D)\ar[d]\\
H^\ast(BT;D)&& H^\ast(BT';D),\ar@<-.5ex>[ll]_{\widetilde{\theta}_1}\ar@<.5ex>[ll]^{\widetilde{\theta}_2}
}$$
then there exists $x\in W'$ such that $\theta_2=\theta_1x$.
\end{enumerate}
\end{theorem}
\begin{proof}
For the first statement of part (1), consider the morphism $H^\ast(BG';\mQ)\to H^\ast(BG;\mQ)$ obtained from $f$ via (c), or equivalently via (d) from the map displayed in part (1). Then \cite[Theorem~1.5(a)]{AM} implies existence of a morphism $H^2(BT';\mQ)\to H^2(BT;\mQ)$ yielding the commutative diagram in part (1) with $\mZ[1/p]$ replaced by $\mQ$. That the latter is induced from a morphism $H^2(BT';\mZ[1/p])\to H^2(BT;\mZ[1/p])$ follows from  \cite[Theorem~1.5(b)]{AM} and the discussion after \cite[Lemma~1.2]{AM}. That the diagram over $\mZ[1/p]$ is commutative follows from faithful flatness of $\mZ[1/p]\to\mQ$.

The case $D=\mQ$ of part (2) is a reformulation of \cite[Theorem~1.7]{AM}. The proof {\it loc. cit.} works for any field of characteristic zero.
The case of integral domains follows from extension of scalars to the field of fractions, using (d).

The second statement of part (1) now follows from part (2) and fact (f), by using $\theta_1=\theta$ and $\theta_2=w\theta$. 
\end{proof}

\begin{corollary}\label{CorAM}
With notation as in Theorem~\ref{ThmAM}, assume that $f$ is a homotopy equivalence. Then the morphism
$$t:\mZ[1/p]\otimes X(T')\;\to\;\mZ[1/p]\otimes X(T)$$
obtained from $\theta$ via the isomorphisms in (f), induces an isomorphism of $\mZ[1/p]$-reflection groups $(W',\mZ[1/p]\otimes X(T'))\to (W,\mZ[1/p]\otimes X(T))$.
\end{corollary}

\subsection{Conclusion of the proof of Theorem~\ref{ThmBG}}

\begin{proof}[Proof of $(b)\Rightarrow (c)$]
Assume first that $p=2$. Then the result follows from Corollary~\ref{CorAM} and Lemma~\ref{Lemp2}(1). Similarly, for $p>2$, by Lemma~\ref{Lemp2}(2) it is sufficient to prove that the extension of scalars along $\mZ[1/p]\to\mZ_2$ of $t$ yields an isomorphism of $\mZ_2$-root data. This allows us to resort to the established theory of $2$-compact groups, see \cite{AG}.

Starting from a homotopy equivalence $f$ as in Theorem~\ref{ThmAM}, we have our $\theta$ from \ref{ThmAM}(1) which induces the isomorphism of reflection groups in Corollary~\ref{CorAM}, and using (g) we have a homotopy equivalence $f_{\widehat{2}}$. We consider the diagram
$$\xymatrix{(BG)_{\widehat{2}}\ar[rr]^{f_{\widehat{2}}}&& (BG')_{\widehat{2}}\\
(BT)_{\widehat{2}}\ar@{-->}[rr]\ar[u]&& (BT')_{\widehat{2}}.\ar[u]
}$$
Now $(BT)_{\widehat{2}}\to (BG)_{\widehat{2}}$ is a maximal torus of the 2-compact group $(BG)_{\widehat{2}}$, in the sense of \cite[Theorem~2.2]{Gr}. By uniqueness of such maximal tori, see {\it loc. cit.}, there exists a homotopy equivalence corresponding to the dashed arrow in the above diagram so that the diagram is commutative up to homotopy.

By \cite[Theorem~3.9]{Su}, this induces 
$$\phi: H^2(BT';\mZ_2)\to H^2(BT;\mZ_2)$$ yielding a commutative diagram
$$\xymatrix{H^\ast(BG;\mZ_2)\ar[d]&& H^\ast(BG';\mZ_2)\ar[ll]\ar[d]\\
H^\ast(BT;\mZ_2)&& H^\ast(BT';\mZ_2)\ar[ll]_{\widetilde{\phi}}.
}$$
By uniqueness in Theorem~\ref{ThmAM}(2) applied to $D=\mZ_2$, we may assume that $\phi$ is actually induced from $\theta$ by extension of scalars $\mZ[1/p]\to\mZ_2$. Finally, the $\mZ_2$-root data of the 2-compact group $(BG)_{\widehat{2}}$, as defined in \cite{AG}, is obtained from the map $BT_{\widehat{2}}\to BG_{\widehat{2}}$ and by construction yields the extension of scalars along $\mZ_2$ of the classical root datum of $G$. The homotopy equivalence $(BG)_{\widehat{2}}\simeq (BG')_{\widehat{2}}$ with commutative diagram therefore indeed implies that our isomorphism of $\mZ[1/p]$-reflection groups extends to an isomorphism of $\mZ_2$-root data.
\end{proof}


\section{Perfected $SL_2$}\label{SecSL2}
Let $k$ be an algebraically closed field of characteristic $p$. For $\lambda$ in $\mN$ or $\mN[1/p]:=\mZ[1/p]\cap\mR_{\ge 0}$, we consider its $p$-adic expansion $\lambda=\sum_i\lambda_ip^i$ with $0\le \lambda_i<p$.

\subsection{Fractal}

\subsubsection{}
For $i,j\in\mN$, denote by $\binom{i}{j}_0$, the zero coefficient of the $p$-adic expansion of the binomial coefficient. By convention, $\binom{i}{j}_0=0$ if $j>i$. For $i,j\in\mN[1/p]$, we set
$$\binom{i}{j}_0\;:=\; \binom{p^li}{p^lj}_0\;\in\;\{0,1,\cdots,p-1\},$$
for some $l\in\mN$ for which $p^li,p^lj\in\mN$. By Lucas' theorem, this definition does not depend on the choice of $l$.

\subsubsection{}
We can describe the weight spaces in simple $(SL_2)_\perf$-modules by
$$\dim_k\bL(n)_{n-2j}\;=\;\begin{cases}
1&\mbox{if}\quad \binom{n}{j}_0\not=0\\
0&\mbox{if}\quad \binom{n}{j}_0=0,
\end{cases}$$
for $n,j\in\mN[1/p]$. In particular, the set 
$$F:=\{(n,i)\,|\,\dim_k\bL(n)_i\not=0\}\subset\mZ[1/p]^{2}\subset\mR^2$$
is a fractal. Concretely, $(a,b)\in F$ if and only if $(p^la,p^lb)\in F$ for all $l\in\mZ$.

The integer points of the fractal $F$ for $p=3$ are displayed in Figure~1.

\subsection{First extensions}

In this section and the next we will apply the short exact sequence
\begin{equation}\label{sesPa}\nabla(\kappa)^{(1)}\hookrightarrow \nabla(p\kappa)\tto \nabla(\kappa-1)^{(1)}\otimes L(p-2),\quad\kappa\in\mN=X_+,
\end{equation}
see \cite[equation~(3)]{Pa}, as well as the isomorphism
\begin{equation}\label{isoPa}\nabla(p\kappa -1)\cong \nabla(\kappa-1)^{(1)}\otimes L(p-1),\quad\kappa\in\mZ_{>0}.\end{equation}

\begin{prop}\label{PropExt1}
For $\blam,\bmu\in \bX_+=\mN[1/p]$, we have
$$\dim\Ext^1(\bL(\blam),\bnab(\bmu))\;=\;\begin{cases}1&\mbox{if there is $i\in\mZ$ with } \lambda_i+\mu_i=p-2 \mbox{ and }\\
&\blam-p^i\lambda_i=\bmu-p^i\mu_i+p^{i+1},\\
0&\mbox{otherwise,}
\end{cases}$$
and
$$\dim\Ext^1(\bL(\blam),\bL(\bmu))\;=\;\begin{cases}1&\mbox{if there is $i\in\mZ$ with } \lambda_i+\mu_i=p-2,\\
& |\lambda_{i+1}-\mu_{i+1}|=1 \mbox{ and }\lambda_j=\mu_j\mbox{ for $j\not\in\{i,i+1\}$,}\\
0&\mbox{otherwise.}
\end{cases}$$
\end{prop}
\begin{proof}
Recall from equation~\eqref{LimNabla} that $\bnab(\blam)$ is a direct limit of pullbacks to $G_{\perf}$ of costandard $G$-modules. Since $\Ext^1(\bL(\blam),-)$ commutes with direct limits in the second argument, we can use Remark~\ref{ExtLim} to conclude\
$$\Ext^1(\bL(\blam),\bnab(\bmu))=\varinjlim \Ext^1(L(p^l\blam),\nabla(p^l\bmu)).$$
The transition maps are given by the composite
$$\Ext^1(L(p^l\blam),\nabla(p^l\bmu))\;\to\; \Ext^1(L(p^l\blam)^{(1)},\nabla(p^l\bmu)^{(1)})\to \Ext^1(L(p^{l+1}\blam),\nabla(p^{l+1}\bmu)).$$
Here, the first map is given by the action of the Frobenius twist, so is injective by Remark~\ref{ExtInj}. The second map comes from the inclusion $\nabla(p^l\bmu)^{(1)}\hookrightarrow \nabla(p^{l+1}\bmu)$. From the description of the cokernel of the inclusion in \eqref{sesPa} it follows that the second map is also injective for $p>2$. For $p=2$ the second map need not be injective, but one shows easily that the composite is still injective.

Assume $p>2$, the case $p=2$ can be proved similarly. It follows quickly from \cite[Corollary~6.2]{Pa}, that for $0\le i<p$
$$\Ext^1(L(pa+i),\nabla(pb+i))\cong\Ext^1(L(a),\nabla(b)),$$
while for $0\le i<p-1$
$$\dim\Ext^1(L(pa+p-2-i),\nabla(pb+i))=\delta_{b+1,a}.$$ 
Together with the block decomposition, this allows by iteration to calculate the first set of extensions.

It follows from Remark~\ref{ExtLim} that
$$\Ext^1(\bL(\blam),\bL(\bmu))=\varinjlim \Ext^1(L(p^l\blam),L(p^l\bmu)),$$
where the transition maps are injective by Remark~\ref{ExtInj}.

Assume $p>2$, the case $p=2$ can be proved similarly. By \cite[Theorem~4.3]{Pa}, we have for $0\le i<p$
\begin{equation}\label{p22}\Ext^1(L(pa+i),L(pb+i))\cong\Ext^1(L(a),L(b)),\end{equation}
while for $0\le i<p-1$
$$\Ext^1(L(pa+i),L(pb+p-2-i))\cong\Hom(L(a),L(b)\otimes L(1)).$$ 
On the other hand, we have
$$\dim\Hom(L(a),L(b)\otimes L(1))=\begin{cases}
1&\mbox{if $b_0=0$ and $a=b+1$}\\
1&\mbox{if $0<b_0<p-1$ and $a=b\pm1$}\\
1&\mbox{if $b_0=p-1$ and $a=b-1$}\\
0&\mbox{otherwise.}
\end{cases}
$$
The cases $b_0<p-1$ follow immediately from the Steinberg tensor product theorem. If $b_0=p-1$, we know by parity that the space is zero unless $a_0<p-1$ in which case we can use symmetry between $a$ and $b$ to reduce to the already known cases.
\end{proof}

\begin{remark}
Equation~\eqref{p22} shows that $\Rep SL_2\to\Rep (SL_2)_{\perf}$ yields isomorphisms on first extensions between simple objects for $p>2$. This is not true for $p=2$.
\end{remark}

\subsection{Costandard modules}
We describe the multiplicities of the simple modules in $\bnab(\blam)$. By Remark~\ref{RemDL}(2) it is sufficient to consider $\blam\in\mN$ (with the case $\blam=0$ trivial).
\begin{prop}\label{PropFilt}
For $\lambda\in \mZ_{>0}$ consider the finite sets
$$E^0(\lambda):=\{\nu\in\mN\,|\, [\nabla(\lambda):L(\nu)]\not=0\}\quad\mbox{and}\quad E^\infty(\lambda):=\{\nu\in\mZ_{>0}\,|\, [\nabla(\lambda-1):L(\nu-1)]\not=0\}.$$
Then, for all $\bmu\in\mN[1/p]$, we have
$$[\bnab(\lambda):\bL(\bmu)]\;=\; \begin{cases}
1&\mbox{ if $\bmu\in E^0(\lambda)$}\\
1&\mbox{ if $\bmu=\nu-\frac{2}{p^i}$ with $\nu\in E^\infty(\lambda)$ and $i>0$}\\
0&\mbox{ otherwise.}
\end{cases}$$
More precisely, $\bnab(\lambda)$ has a filtration $0=M_0\subset M_1\subset M_2\subset\cdots$ with $\cup_iM_i=\bnab(\lambda)$ and 
$$M_1=\nabla(\lambda)[0]\quad\mbox{and}\quad M_{i+1}/M_i\cong \left( \nabla(\lambda-1)^{(i)}\otimes L(p^i-2)\right)[i],\quad\mbox{for $i>0$}.$$
\end{prop}
\begin{proof}
Recall from the proof of Theorem~\ref{ThmSimples} that
$\bnab(\lambda)=\varinjlim \nabla(p^i\lambda)[i]$ where every morphism
in the chain is injective. The corresponding filtration is the desired
one. Indeed the subquotients are given by the cokernel in
\eqref{sesPa} for $\kappa=p^i\lambda$ on which we can apply
iteratively \eqref{isoPa} and the Steinberg tensor product theorem.
\end{proof}

\begin{example}Combining Propositions~\ref{PropExt1} and~\ref{PropFilt} shows:
\begin{enumerate}
\item Consider $0<\lambda<p$, then the socle filtration of $\bnab(\lambda)$ is given by $\soc \bnab(\lambda)=\bL(\lambda)$ and
$$\soc^i\bnab(\lambda)=\bL(\lambda-\frac{2}{p^i}),\quad i>0.$$
\item  The socle filtration of $\bnab(2p-1)$ is given by $\soc\bnab(2p-1)=\bL(2p-1)$, $\soc^1\bnab(2p-1)=\bL(2p-1-2/p)$ and
$$\soc^{i+1}\bnab(2p-1)\cong \bL(2p-1-\frac{2}{p^{i+1}})\oplus \bL(1-\frac{2}{p^{i}}), \quad i>0.$$
\end{enumerate}
\end{example}

\begin{remark}\label{RemSymP}
We can explicitly realise $\bnab(\blam)$ as the space of `degree $\blam$' elements in $k[x^{1/p^\infty},y^{1/p^\infty}]$, that is the span of $\{x^{\bmu}y^{\bnu}\,|\,\bmu,\bnu\in \mN[1/p], \;  \bmu+\bnu=\blam\}$.
\end{remark}

\subsection{Line bundle cohomology}

We consider the representations $\bW(\blam)$, $\blam\in\mN[1/p]\backslash \{0\}=\bX_{++}$, from \ref{WeylType}.

\begin{prop}\label{PropFilt2} Recall the finite sets $E^0,E^\infty$ from Proposition~\ref{PropFilt}.
For $\lambda\in \mZ_{>0}$ and $\bmu\in\mN[1/p]$, we have
$$[\bW(\lambda):\bL(\bmu)]\;=\; \begin{cases}
1&\mbox{ if $\bmu\in E^0(\lambda-2)$, (with $E^0(-1):=\varnothing$)}\\
1&\mbox{ if $\bmu=\nu-\frac{2}{p^i}$ with $\nu\in E^\infty(\lambda)$ and $i>0$}\\
0&\mbox{ otherwise.}
\end{cases}$$
More precisely, $\bW(\lambda)$ has a filtration $0=M_0\subset M_1\subset M_2\subset\cdots$ with $\cup_iM_i=\bW(\lambda)$ and (with convention $\Delta(-1)=0$)
$$M_1=\Delta(\lambda-2)[0]\quad\mbox{and}\quad M_{i+1}/M_i\cong \left( \Delta(\lambda-1)^{(i)}\otimes L(p^i-2)\right)[i],\quad\mbox{for $i>0$}.$$
\end{prop}
\begin{proof}
Using \v{C}ech cohomology (see proof of Lemma~\ref{LemLine}), it follows easily that the morphisms in the directed system in \ref{WeylType} are injective. The result then follows as in the proof of Proposition~\ref{PropFilt}, by now using \cite[(3)]{Pa} for $i=p-2$.\end{proof}

\begin{remark}
It follows that in the Grothendieck group of $\Rep (SL_2)_{\perf}$, we have
$$[\bnab(\lambda)]-[\bW(\lambda)]\;=\;[\Delta(\lambda)[0]]-[\Delta(\lambda-2)[0]].$$
\end{remark}

\begin{example}
We have a short exact sequence
$$0\to \bL(1)\to \bnab(1)\to \bW(1)\to 0.$$
For $1<\lambda<p-1$, we have
$$\bnab(\lambda)/\bL(\lambda)\;\cong\;\bW(\lambda)/\bL(\lambda-2).$$
\end{example}


\subsection*{Acknowledgement}

The research was partly supported by ARC project DP210100251. The
authors thank Brian Conrad for useful discussions, Pavel Etingof for
help with Example~\ref{ExGm}, Anthony Henderson for suggesting to
include isogenies of perfectly reductive groups { and
  Steve Donkin for providing Example \ref{ex:Donkin} and pointing out \cite{Wang}.}

\end{document}